\documentclass[a4paper,10pt,american]{amsart}
\usepackage{amssymb}
\usepackage{dsfont}
\usepackage{stmaryrd}
\usepackage{esint}

\usepackage{xcolor}
\usepackage[colorlinks,pdfpagelabels,pdfstartview = FitH,bookmarksopen
= true,bookmarksnumbered = true,linkcolor = blue,plainpages =
false,hypertexnames = false,citecolor = red,pagebackref=false]{hyperref}

\allowdisplaybreaks

\newtheorem{theorem}{Theorem}
\newtheorem{lemma}[theorem]{Lemma}

\theoremstyle{definition}
\newtheorem{definition}[theorem]{Definition}

\theoremstyle{remark}

\numberwithin{theorem}{section}
\numberwithin{equation}{section}

\newcommand{\N}{\mathbb{N}}
\newcommand{\R}{\mathbb{R}}

\renewcommand{\L}{\mathcal{L}}

\newcommand{\dx}{\mathrm{d}x}

\newcommand{\dt}{\mathrm{d}t}
\newcommand{\ds}{\mathrm{d}s}

\newcommand{\dsigma}{\mathrm{d}\sigma}

\renewcommand{\epsilon}{\varepsilon}
\renewcommand{\rho}{\varrho}

\DeclareMathOperator{\Div}{div}

\newcommand{\wto}{\rightharpoondown}
\newcommand{\wsto}{\overset{\raisebox{-1ex}{\scriptsize $*$}}{\rightharpoondown}}

\def\Xint#1{\mathchoice
    {\XXint\displaystyle\textstyle{#1}}%
    {\XXint\textstyle\scriptstyle{#1}}%
    {\XXint\scriptstyle\scriptscriptstyle{#1}}%
    {\XXint\scriptscriptstyle\scriptscriptstyle{#1}}%
    \!\int}
\def\XXint#1#2#3{\setbox0=\hbox{$#1{#2#3}{\int}$}
    \vcenter{\hbox{$#2#3$}}\kern-0.5\wd0}
\def\bint{\Xint-}
\def\dashint{\Xint{\raise4pt\hbox to7pt{\hrulefill}}}

\def\Xiint#1{\mathchoice
    {\XXiint\displaystyle\textstyle{#1}}%
    {\XXiint\textstyle\scriptstyle{#1}}%
    {\XXiint\scriptstyle\scriptscriptstyle{#1}}%
    {\XXiint\scriptscriptstyle\scriptscriptstyle{#1}}%
    \!\iint}
\def\XXiint#1#2#3{\setbox0=\hbox{$#1{#2#3}{\iint}$}
    \vcenter{\hbox{$#2#3$}}\kern-0.5\wd0}
\def\biint{\Xiint{-\!-}}

\subjclass[2020]{35A01, 35K61, 35K86, 49J40}

\begin{document}
\title[The bounded slope condition for time-dependent integrands]{The bounded slope condition for parabolic equations with time-dependent integrands}
\date{\today}

\author[L.~Sch\"atzler]{Leah Sch\"atzler}
\address{Leah Sch\"atzler\\
Fachbereich Mathematik, Paris-Lodron Universit\"at Salzburg\\
Hellbrunner Straße 34, 5020 Salzburg, Austria}
\email{leahanna.schaetzler@plus.ac.at}

\author[J.~Siltakoski]{Jarkko Siltakoski}
\address{Jarkko Siltakoski\\
Department of Mathematics and Statistics, University of Jyv\"as-\\kyl\"a
P.O.Box 35, FIN-40014, Finland}
\email{jarkko.j.m.siltakoski@jyu.fi}

\begin{abstract}
In this paper, we study the Cauchy-Dirichlet problem
\begin{equation*}
	\left\{
	\begin{array}{ll}
		\mbox{$\partial_t u - \operatorname{div} \left( D_\xi f(t, Du)\right) = 0$ } & \mbox{\quad in $\Omega_T$}, \\[5pt]
		\mbox{$u = u_o$} & \mbox{\quad on $\partial_{\mathcal{P}} \Omega _T$},\\[5pt]
	\end{array}
	\right.
\end{equation*}
where $\Omega \subset \mathbb{R}^n$ is a convex domain, $f:[0,T]\times\mathbb{R}^n \rightarrow \mathbb{R}$ is $L^1$-integrable in time and convex in the second variable. Assuming that the initial and boundary datum $u_o:\overline{\Omega}\rightarrow \mathbb{R}$ satisfies the bounded slope condition, we prove the existence of a unique variational solution that is Lipschitz continuous in the space variable.
\end{abstract}

%***************************************************************************

\maketitle

\section{Introduction and results}
It follows from classical theory \cite{Haar, Hartman-Nirenberg, Hartman-Stampacchia, Miranda, Stampacchia} (see also \cite[Chapter 1]{Giusti}) that any variational functional $F:W^{1,\infty}(\Omega) \rightarrow \mathbb{R}$ of the form
\begin{equation*}
	F(v) := \int_\Omega f(Dv) \, \dx,
\end{equation*}
where $f:\mathbb{R}^n \rightarrow \mathbb{R}$ is convex and $\Omega \subset \mathbb{R}^n$ is a convex domain, admits a unique Lipschitz continuous minimizer in the class $\{v \in W^{1,\infty} (\Omega): v=v_o \text{ on } \partial \Omega\}$ provided that the boundary datum $v_o$  satisfies the bounded slope condition (see Definition \ref{def:bounded_slope}).
Modern elliptic results involving one-sided bounded slope conditions or more general integrands include for example \cite{Bousquet07, Bousquet10,Bousquet-Brasco16} and \cite{Giannetti-Treu22, Mariconda-Treu02, Mariconda-Treu07, Mariconda-Treu11, Don-Lussardi-Pinamonti-Treu22}.

Surprisingly, while Hardt and Zhou \cite[Chapter 4]{Hardt-Zhou} used the bounded slope condition in a regularity argument in a time-dependent setting involving functionals with linear growth, an evolutionary analogue of the above stationary theorem was established only rather recently by  Bögelein, Duzaar, Marcellini and Signoriello \cite{BDMS}. They considered the Cauchy-Dirichlet problem
\begin{equation*}
	\left\{
	\begin{array}{ll}
		\mbox{$\partial_t u - \operatorname{div} \left( D f(Du)\right) = 0$ } & \mbox{\quad in $\Omega_T$}, \\[5pt]
		\mbox{$u = u_o$} & \mbox{\quad on $\partial_{\mathcal{P}} \Omega _T$},\\[5pt]
	\end{array}
	\right.
\end{equation*}
where $\Omega_T := \Omega \times (0,T)$ with $\Omega \subset \R^n$ and $T \in (0,\infty]$ denotes a space-time cylinder and $\partial_\mathcal{P} \Omega_T := \partial \Omega \times (0,T) \cup (\overline \Omega \times \{0\})$ its parabolic boundary. Given a Lipschitz continuous initial and boundary datum $u_o$ that satisfies the bounded slope condition, in \cite{BDMS} it was proven that the above problem admits a unique variational solution that is globally Lipschitz continuous with respect to the spatial variables.
Moreover, if the integrand $f$ fulfills an additional $p$-coercivity condition with some $p>1$, Bögelein and Stanin \cite{Bogelein-Stanin} obtained the local Lipschitz continuity of variational solutions in space and time under the weaker one-sided bounded slope condition, which means that the $u_o$ is convex and Lipschitz continuous.
Further, global continuity of $u$ was proven in the case that $\Omega$ is uniformly convex.

For the same class of integrands and merely convex domains $\Omega$, Stanin \cite{Stanin} showed that variational solutions are still globally Hölder continuous even if the convexity assumption on $u_o$ is dropped. Equations with lower-order terms were considered by Rainer,  Siltakoski and Stanin \cite{Rainer-Siltakoski-Stanin} who extended a stationary Haar-Rado type theorem by Mariconda and Treu \cite{Mariconda-Treu11} to the parabolic problem
\begin{equation*}
	\left\{
	\begin{array}{ll}
		\mbox{$\partial_t u - \operatorname{div} \left( D f(Du)\right)  + D_u g(x, u) = 0$ } & \mbox{\quad in $\Omega_T$}, \\[5pt]
		\mbox{$u = u_o$} & \mbox{\quad on $\partial_{\mathcal{P}} \Omega _T$},\\[5pt]
	\end{array}
	\right.
\end{equation*}
where $f$ is convex and $p$-coercive with some $p>1$ and the lower-order term $g$ satisfies a technical condition, in particular convexity with respect to $u$. As a corollary, the authors in \cite{Rainer-Siltakoski-Stanin} obtained the global Lipschitz continuity with respect to the spatial variables of variational solutions under the classical two-sided bounded slope condition provided that $f \in C^2$ is uniformly convex in a suitable sense.

The objective of the present paper is to extend the result of \cite{BDMS} to include time-dependent integrands.
In order to focus on the novelty and to include integrands $f$ with linear growth, we consider the classical bounded slope condition and avoid lower-order terms.
We are concerned with parabolic partial differential equations of the form
\begin{equation}
	\partial_t u - \Div (D_\xi f(t, Du)) = 0
	\quad \text{in } \Omega_T,
	\label{eq:pde}
\end{equation}
where $\Omega \subset \mathbb{R}^n$ is a convex domain,  $T \in (0,\infty]$, and the integrand $f \colon [0,T] \times \R^n \to \R$ satisfies the following assumptions:
\begin{equation}
	\left\{
	\begin{array}{l}
		\mbox{$\xi \mapsto f(t,\xi)$ is convex for a.e.~$t \in [0,T]$,} \\[5pt]
		\mbox{$t \mapsto f(t,\xi) \in L^1(0,\tau)$ for all $\xi \in \R^n$ and $\tau \in (0, T] \cap \R$.} \\[5pt]
	\end{array}
	\right.
	\label{eq:integrand}
\end{equation}
In particular, $f$ is a Carath\'eodory function and for any $L>0$ and $\tau \in (0, T] \cap \R$ the map $t \mapsto \max_{|\xi| \leq L} |f(t,\xi)|$ belongs to $L^1(0, \tau)$ (see Section \ref{sec:convex} below).
Therefore, for any $\tau \in (0, T] \cap \R$ and $V \in L^\infty(\Omega_T,\R^n)$ we have that 
$$
	\iint_{\Omega_T} |f(t,V)| \,\dx\dt < \infty.
$$
We emphasize that $t \mapsto f(t, \xi)$ is neither assumed to be continuous nor weakly differentiable.

Examples of admissible integrands are functionals with linear growth such as the area integrand $f(\xi) = \sqrt{1 + |\xi|^2}$, integrands with exponential growth like $f(\xi) = \exp(|\xi|^2)$, Orlicz type functionals such as $f(\xi) = |\xi| \log(1 + |\xi|)$ and time-dependent combinations thereof like $f(t,\xi) = \chi_{[0,t_o]} f_1(\xi) + \chi_{(t_o, T]} f_2(\xi)$ or more general $f(t,\xi) = \sum_{i=1}^m a_i(t) f_i(\xi)$ for functions $a_i \in L^1(0,T)$, $i=1,\ldots,m$.

In the present paper, we work with variational solutions in the spirit of Lichnewsky and Temam \cite{Lichnewsky-Temam}, see also for example \cite{BDM, BDM14, BDMS}. We consider the following class of functions that are Lipschitz continuous in space
\begin{equation*}
	K^{\infty} := \{ v \in L^{\infty}(\Omega_T) \cap C^0 ([0, T]; L^2(\Omega)): Dv \in L^\infty(\Omega_T, \mathbb{R}^n) \}.
\end{equation*}
Further, we denote the subclass related to time-independent boundary values $u_o \in W^{1, \infty} (\Omega)$ by
\begin{equation*}
	K^{\infty}_{u_o} := \{ v\in K^L(\Omega_T): v = u_o \text{ on the lateral boundary } \partial \Omega \times (0,T) \}.
\end{equation*}

\begin{definition}[Variational solutions]
\label{def:unconstrainded_solution}
Assume that $f \colon [0,T] \times \R^n \to \R$ satisfies \eqref{eq:integrand}
and consider a boundary datum $u_o \in W^{1,\infty}(\Omega)$.
In the case $T \in (0,\infty)$ a map $u \in K^\infty_{u_o}(\Omega_T)$ is called a \emph{variational solution}
to the Cauchy-Dirichlet problem associated with \eqref{eq:pde} and $u_o$ in $\Omega_T$
if and only if the variational inequality
\begin{align}
	\iint_{\Omega_T} f(t,Du) \,\dx\dt
	&\leq
	\iint_{\Omega_T} \partial_t v (v - u) + f(t,Dv) \,\dx\dt \label{eq:var_ineq_unconstrained} \\
	&\phantom{=}
	+ \tfrac12 \|v(0) - u_o\|_{L^2(\Omega)}^2
	- \tfrac12 \|(v - u)(T)\|_{L^2(\Omega)}^2
	\nonumber
\end{align}
holds true for any comparison map $v \in K^\infty_{u_o}(\Omega_T)$
with $\partial_t v \in L^2(\Omega_T)$.
If $T=\infty$ and $u \in K^\infty_{u_o}(\Omega_\infty)$ is a variational solution
in $\Omega_\tau$ for any $\tau \in (0,\infty)$, $u$ is called a \emph{global variational solution}
or variational solution in $\Omega_\infty$ to the Cauchy-Dirichlet problem
associated with \eqref{eq:pde} and $u_o$.
\end{definition}

Our main result concerning the existence of variational solutions
which are Lipschitz continuous with respect to the spatial variables
can be formulated as follows.
\begin{theorem}
\label{thm:main_existence}
Let $\Omega \subset \R^n$ be an open, bounded and convex set and $T \in (0, \infty]$.
Assume that $f \colon [0,T] \times \R^n \to \R$ satisfies hypotheses
\eqref{eq:integrand}.
Further, let $u_o \in W^{1,\infty}(\Omega_T)$ denote a boundary datum such that
the bounded slope condition with some positive constant $Q$
(see Definition \ref{def:bounded_slope} below)
is fulfilled for $U_o := \left. u_o \right|_{\partial\Omega}$.
Then, there exists a unique variational solution $u$ to the Cauchy-Dirichlet problem
associated with \eqref{eq:pde} and $u_o$ in $\Omega_T$.
Moreover, $u$ satisfies the gradient bound
\begin{equation}
	\|Du\|_{L^\infty(\Omega_T,\R^n)}
	\leq
	\max\{ Q, \|Du_o\|_{L^\infty(\Omega,\R^n)} \}.
	\label{eq:gradient_bound}
\end{equation}
\end{theorem}

Furthermore, we show that variational solutions to \eqref{eq:pde} are weak solutions and consequently, they are $1/2$-Hölder continuous in time provided that the map $\xi \mapsto f(t, \xi)$ is $C^1$ and uniformly locally Lipschitz in the following sense: For each $L>0$, there exists a constant $M_L>0$ such that
\begin{equation}
\label{eq:uniform_lipschitz}
	\sup_{t\in(0,T)} |D_\xi f(t, \xi)| < M_L \quad \text{for all} \quad \xi \in B_L(0).
\end{equation}

\begin{theorem}
\label{thm:regularity}
Suppose that the assumptions of Theorem \ref{thm:main_existence} hold. Moreover, assume that the mapping $\xi \mapsto f(t, \xi)$ is in $C^1({\mathbb{R}}^n)$ for almost all $t \in (0,T)$ and satisfies \eqref{eq:uniform_lipschitz}. Then the unique variational solution $u$ to the Cauchy-Dirichlet problem associated with \eqref{eq:pde} and $u_o$ is a weak solution. Further, it is contained in the space of H\"older continuous functions $C^{0;1,1/2} ({\overline \Omega}_T)$.
\end{theorem}

To prove Theorem \ref{thm:main_existence}, we may assume without a loss of generality that $T < \infty$, see the beginning of Section  \ref{sec:unconstrained_for_general}.
The proof is divided into three parts. We first assume that the integrand is suitably regular and in particular has a weak derivative with respect to the time variable. Then the method of minimizing movements yields a solution $u$ to the so called \textit{gradient constrained obstacle problem} to \eqref{eq:pde}, where the $L^\infty$-norm of the gradients of the solution and the comparison maps are bounded by a fixed constant $L \in (0, \infty)$. Moreover, the regularity assumption on $f$ ensures that $u$ has a weak time derivative in $L^2(\Omega_T)$.

Next, under the same regularity assumptions on $f$ as in the first step, a standard argument exploiting the bounded slope condition and the maximum principle yields the uniform gradient bound \eqref{eq:gradient_bound} for $u$. 
Choosing $L$ large enough, this in turn allows us to deduce that $u$ is in fact already a solution to the unconstrained problem in the sense of Definition \ref{def:unconstrainded_solution}.

To deal with a general integrand $f$, we consider its Steklov average $f_\varepsilon$. Since $f_\varepsilon$ admits a weak time derivative, by the results mentioned in the preceding paragraph there exists a solution $u_\varepsilon$ to the Cauchy-Dirichlet problem associated with $f_\varepsilon$ in the sense of Definition \ref{def:unconstrainded_solution}.
Moreover, since for each $\varepsilon>0$ the solution $u_\varepsilon$ satisfies the gradient bound \eqref{eq:gradient_bound} and $u_\varepsilon$ = $u_o$ on $\partial\Omega \times (0,T)$, there exists a limit map $u \in L^\infty (\Omega_T)$ such that $u_\varepsilon \rightarrow u$ uniformly and $Du_\varepsilon \rightarrow Du$ weakly$^\ast$ up to a subsequence as $\varepsilon \downarrow 0$. This allows us to conclude that $u$ is a variational solution in the sense of Definition \ref{def:unconstrainded_solution}, finishing the proof of Theorem \ref{thm:main_existence}.

The proof of Theorem \ref{thm:regularity} is similar to the one found in \cite[Chapter 8]{BDMS}. The $C^1$ assumption on the integrand ensures the validity of the weak Euler-Lagrange equation, which lets us apply the argument from \cite[p23-24]{BDMi13} to prove a Poincar\'e inequality for variational solutions. The H\"older continuity then follows from the Campanato space characterization of H\"older continuity by Da Prato \cite{DaPrato}.

The paper is organized as follows. Section \ref{sec:preliminaries} contains preliminary definitions and basic observations about the integrand. In Section \ref{sec:properties} we prove certain properties of variational solutions that are required in later sections, including the comparison and maximum principles. Under additional regularity assumptions on $f$ we use the method of minimizing movements to prove the existence of variational solutions to the gradient constrained problem in Section \ref{sec:existence_for_constrained} and in Section \ref{sec:unconstrained_bsc} we consider the unconstrained problem. Finally, in Section \ref{sec:unconstrained_for_general} we consider general integrands and finish the proof of Theorem \ref{thm:main_existence} and H\"older continuity in time is proven in Section \ref{sec:time} under additional regularity assumptions.

\section{Preliminaries}\label{sec:preliminaries}
\subsection{Notation}
Throughout the paper, for $p \in [1,\infty]$ and $m \in \N$ the space $L^p(\Omega,\R^m)$ denotes the usual Lebesgue space (we omit $\R^m$ if $m=1$) and $W^{1,p}(\Omega)$ and $W^{1,p}_0(\Omega)$ denote the usual Sobolev spaces.
If $\Omega$ is a bounded Lipschitz domain, $W^{1,\infty}(\Omega)$ can be identified with the space $C^{0,1}(\overline{\Omega})$ of functions $v \colon \Omega \to \R$ that are Lipschitz continuous (with Lipschitz constant $[v]_{0,1} = \|Dv\|_{L^\infty(\Omega,\R^n)}$) up to the boundary of $\Omega$.
Note that in particular any convex set has a Lipschitz continuous boundary, since convex functions are locally Lipschitz \cite[Corollary 2.4]{Ekeland-Temam}.
Further, for a Banach space $X$ and an integrability exponent $p \in [1,\infty]$ we write $L^p(0,T;X)$ for the space of Bochner measurable functions $v \colon [0,T] \to X$ with $t \mapsto \|v(t)\|_X \in L^p(0,T)$.
Moreover, $C^0([0,T];X)$ is defined as the space of the continuous functions $v \colon [0,T] \to X$.
For maps $v$ defined in $\Omega_T$ we also use the short notation $v(t)$ for the partial map $x \mapsto v(x,t)$ defined in $\Omega$.
Finally, for a set $A \subset \R^m$, the characteristic function $\chi_A \colon \R^m \to \{0,1\}$ is given by $\chi_A(x) = 1$ if $x \in A$ and $\chi_A(x) = 0$ else.
%Finally, consider a domain $\Omega \subset \R^n$ and recall that a function $v \colon \Omega \to \R$ is Lipschitz continuous in $\Omega$ with respect to the metric $d \colon \Omega \times \Omega \to [0,\infty)$ on $\R^n$ if its Lipschitz constant
%$$
%	[v]_{0,1;\Omega} =
%	\sup_{x \neq y, x,y \in \Omega} \frac{|v(x) - v(y)|}{d(x,y)}
%$$
%is finite.
%Note that a map $v \in L^1_{loc}(\Omega)$ admits a weak derivative $Dv \in L^\infty(\Omega,\R^n)$ if and only if $v$ is Lipschitz continuous with respect to the inner metric $d_\Omega$.

\subsection{Bounded slope condition}
In the proof of the existence result in Section~\ref{sec:unconstrained_bsc} it is crucial that there exist affine comparison functions below and above the initial/boundary datum $u_o$ coinciding with $u_o$ in a point $x_o \in \partial\Omega$.
This is ensured by applying the following bounded slope condition to $\left. u_o \right|_{\partial\Omega}$.
\begin{definition}
\label{def:bounded_slope}
A function $U \colon \partial\Omega \to \R$ satisfies the \emph{bounded slope condition}
with constant $Q>0$ if for any $x_o \in \partial\Omega$ there exist two affine functions
$w_{x_o}^\pm \colon \R^n \to \R$ with Lipschitz constants $[w_{x_o}^\pm]_{0,1} \leq Q$
such that
$$
	\left\{
	\begin{array}{l}
		\mbox{$w_{x_o}^-(x) \leq U(x) \leq w_{x_o}^+(x)$ for any $x \in \partial\Omega$,} \\[5pt]
		w_{x_o}^-(x_o) = U(x_o) = w_{x_o}^+(x_o).
	\end{array}
	\right.
$$
\end{definition}
Note that unless $U$ itself is affine, the convexity of $\Omega$ is necessary for the bounded slope condition to hold.
Note that even strict convexity of $\Omega$ is not sufficient for general $U$, since the boundary can become ``too flat".
However, we know that for a uniformly convex, bounded $C^2$-domain $\Omega$ and $v \in C^2(\R^n)$ the restriction $U= \left. v \right|_{\partial\Omega}$ fulfills the bounded slope condition.
For more details, we refer to \cite{Giusti,Miranda}.
On the other hand, in the parabolic setting the following example is relevant: Consider a convex domain $\Omega$ with flat parts (such as a rectangle) and a Lipschitz continuous function $u_o$ that vanishes at the boundary of $\Omega$; i.e.~we prescribe zero lateral boundary values, but the initial datum is not necessarily identical to zero.

We need the following lemma from \cite[Lemma 2.3]{BDMS}. It states that if $u_o$ is Lipschitz and $u_o|_{\partial\Omega}$ satisfies the bounded slope condition, then $u_o$ can be squeezed between two affine functions that touch $u_o$ at a given boundary boundary point and the Lipschitz constant of these affine functions is bounded by either the Lipschitz constant of $u_o$ or the constant in the bounded slope condition.
\begin{lemma}
\label{lem:bounded_slope}
Let $u_o \in C^{0,1}(\overline{\Omega})$ with Lipschitz constant $[u_o]_{0,1} \leq Q_1$ such that the restriction $U := \left. u_o \right|_{\partial\Omega}$ satisfies the bounded slope condition with constant $Q_2$.
Then for any $x_o \in \partial\Omega$ there exist two affine functions $w_{x_o}^\pm$ with $[w_{x_o}^\pm]_{0,1} \leq \max\{Q_1, Q_2\}$ such that
$$
	\left\{
	\begin{array}{l}
		\mbox{$w_{x_o}^-(x) \leq u_o(x) \leq w_{x_o}^+(x)$ for any $x \in \overline{\Omega}$,} \\[5pt]
		w_{x_o}^-(x_o) = u_o(x_o) = w_{x_o}^+(x_o).
	\end{array}
	\right.
$$
\end{lemma}

\subsection{Dominating functions for the integrand}
\label{sec:convex}
Observe that for any $L>0$ the map $t \mapsto \max_{|\xi| \leq L} f(t,\xi)$ is measurable, since we have that $\max_{|\xi| \leq L} f(t,\xi) = \max_{\xi \in B_L(0) \cap \mathds{Q}^n} f(t, \xi)$ and the maximum of countably many measurable functions is measurable.
The same holds true for $t \mapsto \min_{|\xi| \leq L} f(t,\xi)$.
In the following lemma, we show that they are contained in $L^1(0,T)$.
\begin{lemma}
\label{lem:equiL1}
Let $T\in(0, \infty)$ and assume that $f \colon [0,T] \times \R^n \to \R$ satisfies \eqref{eq:integrand}.
Then, for any $L>0$ there exists a function $g_L \in L^1(0,T)$ such that
\begin{equation}
	|f(t,\xi)| \leq g_L(t)
	\quad \text{for all $t \in (0,T)$ and $\xi \in B_L(0)$.}
	\label{eq:dominating_f}
\end{equation}
\end{lemma}
\begin{proof}
First, we show that for any $L>0$, we have that
\begin{equation}
	t \mapsto \max_{|\xi| \leq L} f(t,\xi) \in L^1(0,T).
	\label{eq:max_aux}
\end{equation}
To this end, fix $\xi_1, \ldots, \xi_{n+1} \in \R^n$ such that the closed ball $B_L(0)$ is a subset of the simplex
$$
	\Delta :=
	\left\{ \xi \in \R^n : \xi = \sum_{i=1}^{n+1} \lambda_i \xi_i \text{ with } 0 \leq \lambda_i \leq 1, i=1,\ldots,n+1, \sum_{i=1}^{n+1} \lambda_i = 1 \right\}.
$$
Note that for any $t \in [0,T]$ such that $\R^n \ni \xi \mapsto f(t,\xi)$ is convex,  the mapping $\xi \mapsto f(t,\xi)$ attains its maximum in one of the points $\xi_1, \ldots, \xi_{n+1}$.
Hence, for a.e.~$t$ we obtain that
\begin{align*}
	f(t,0)
	&\leq
	\max_{|\xi| \leq L} f(t,\xi)
	\leq
	\sum_{i=1}^{n+1} |f(t,\xi_i)|.
\end{align*}

%\textcolor{red}{This seems to be a well-known fact. Do you think that we need a reference?}
%For any $t \in [0,T]$ such that $\xi \mapsto f(t,\xi)$ is convex, there exists $\xi_{max} \in B_L(0)$ with $f(t,\xi_{max}) = \max_{|\xi| \leq L} f(t,\xi)$.
%Since $B_L(0) \subset \Delta$, there exist $\lambda_1, \ldots, \lambda_{n+1} \in [0,1]$ such that $\xi_{max} = \sum_{i=1}^{n+1} \lambda_i \xi_i$.
%By the convexity of $\xi \mapsto f(t,\xi)$, we obtain that
%\begin{align*}
%	f(t,0)
%	&\leq
%	\max_{|\xi| \leq L} f(t,\xi)
%	=
%	f(t,\xi_{max})
%	=
%	f \left(t, \sum_{i=1}^{n+1} \lambda_i \xi_i \right)
%	\leq
%	\sum_{i=1}^{n+1} \lambda_i f(t,\xi_i)
%	\leq
%	\sum_{i=1}^{n+1} |f(t,\xi_i)|.
%\end{align*}
Since the maps $t \mapsto f(t,0)$ and $t \mapsto f(t,\xi_i)$, $i=1,\ldots,n+1$, belong to $L^1(0,T)$ by \eqref{eq:integrand}$_2$, this implies \eqref{eq:max_aux}.

Next, we fix $L>0$ and prove
\begin{equation}
	t \mapsto \min_{|\xi| \leq L} f(t,\xi) \in L^1(0,T).
	\label{eq:min_aux}
\end{equation}
Consider $t \in [0,T]$ such that $\xi \mapsto f(t,\xi)$ is convex.
Then, there exist $\xi_{min}, \xi_{max} \in B_L(0)$ such that $f(t,\xi_{min}) = \min_{|\xi| \leq L} f(t,\xi)$ and $f(t,\xi_{max}) = \max_{|\xi| \leq L} f(t,\xi)$.
Assume that $\xi_{min} \neq \xi_{max}$ (otherwise, $\xi \mapsto f(t,\xi)$ is constant in $B_L(0)$ and thus $f(t,0) = \min_{|\xi|\leq L)} f(t, \xi)$).
First, note that for $C := \tfrac{1}{2L} (f(t,\xi_{max}) - f(t,\xi_{min})) \in (0,\infty)$, we find that
%$$
%	f(t,\xi_{max}) - f(t,\xi_{min})
%	\geq
%	C |\xi_{max} - \xi_{min}|,
%$$
%which is equivalent to
$$
	f(t,\xi_{min})
	\leq
	f(t,\xi_{max}) - C |\xi_{max} - \xi_{min}|.
$$
Furthermore, since $\xi \mapsto f(t,\xi)$ is convex in $\R^n$, its subdifferential at $\xi_{max}$ is non-empty \cite[Proposition 5.2]{Ekeland-Temam}, i.e.~there exists $\eta = \eta(\xi_{max}) \in \R^n$ such that
$$
	f(t,\xi)
	\geq
	f(t,\xi_{max}) + \eta \cdot (\xi - \xi_{max})
$$
for any $\xi \in \R^n$.
In particular, we have that
$$
	f(t,\xi_{min})
	\geq
	f(t,\xi_{max}) + \eta \cdot (\xi_{min} - \xi_{max})
	=
	f(t,\xi_{max}) + \cos(\alpha) |\eta| |\xi_{min} - \xi_{max}|,
$$
where $\alpha$ denotes the angle between $\eta$ and $\xi_{min} - \xi_{max}$.
Together, the preceding two inequalities imply that
$$
	\cos(\alpha) |\eta| \leq -C.
$$
Next, choose $s>1$ such that $\xi_o := \xi_{min} + s(\xi_{max} - \xi_{min}) \in \partial B_{L+1}(0)$.
Note that the vector $\xi_o - \xi_{max} = (1-s)(\xi_{min} - \xi_{max})$ points in the opposite direction as $\xi_{min} - \xi_{max}$.
Therefore, the angle between $\eta$ and $\xi_o - \xi_{max}$ is $\pi - \alpha$.
Using the facts that $\cos(\pi - \alpha) = -\cos(\alpha)$ and $|\xi_o - \xi_{max}| \geq 1$, the preceding inequality and the definition of $C$, we conclude that
\begin{align*}
	\max_{|\xi| \leq L+1} f(t,\xi)
	\geq
	f(t,\xi_o)
	&\geq
	f(t,\xi_{max}) + \eta \cdot (\xi_o - \xi_{max}) \\
	&=
	f(t,\xi_{max}) - \cos(\alpha) |\eta| |\xi_o - \xi_{max}| \\
	&\geq
	f(t,\xi_{max}) + C \\
	&=
	\max_{|\xi| \leq L} f(t,\xi) + \tfrac{1}{2L} (\max_{|\xi| \leq L} f(t,\xi)) - \min_{|\xi| \leq L} f(t,\xi))).
\end{align*}
This is equivalent to
$$
	(2L+1) \max_{|\xi| \leq L} f(t,\xi) - 2L \max_{|\xi| \leq L+1} f(t,\xi)
	\leq	
	\min_{|\xi| \leq L} f(t,\xi)
	\leq
	\max_{|\xi| \leq L} f(t,\xi),
$$
which holds for almost every $t \in [0, T]$. Since we have already shown that $t \mapsto \max_{|\xi| \leq L} f(t,\xi)$ and $t \mapsto \max_{|\xi| \leq L+1} f(t,\xi)$ are contained in $L^1(0,T)$, the preceding inequality proves \eqref{eq:min_aux}.
The claim of Lemma \ref{lem:equiL1} follows by combining \eqref{eq:max_aux} and \eqref{eq:min_aux}.
\end{proof}

\subsection{Lower semicontinuity}
In the course of the paper we will need Le following result on the lower semicontinuity of integrals involving $f$ with respect to the weak$^\ast$ topology of $L^\infty(\Omega_T,\R^n)$.

\begin{lemma}
\label{lem:lower_semicontinuity}
Let $\Omega \subset \R^n$ be a bounded open set and $0<T<\infty$.
Assume that $f \colon [0,T] \times \R^n \to \R$ satisfies \eqref{eq:integrand}.
Then, for any sequence $(V_i)_{i \in \N} \subset L^\infty(\Omega_T,\R^n)$ and $V \in L^\infty(\Omega_T,\R^n)$ such that $V_i \wsto V$ weakly$^\ast$ in $L^\infty(\Omega_T,\R^n)$ as $i \to \infty$ we have that
$$
	\iint_{\Omega_T} f(t,V) \,\dx\dt
	\leq
	\liminf_{i \to \infty} \iint_{\Omega_T} f(t,V_i) \,\dx\dt.
$$
\end{lemma}
\begin{proof}
Consider an arbitrary sequence $(V_i)_{i \in \N} \subset L^\infty(\Omega_T,\R^n)$ and a limit map $V \in L^\infty(\Omega_T,\R^n)$ such that $V_i \wsto V$ weakly$^\ast$ in $L^\infty(\Omega_T,\R^n)$ as $i \to \infty$.
First, note that $(V_i)_{i \in \N}$ is bounded in $L^\infty(\Omega_T,\R^n)$ and set $M := \sup_{i \in \N} \| V_i \|_{L^\infty(\Omega_T,\R^n)} \geq \| V \|_{L^\infty(\Omega_T,\R^n)}$.
We find that
$$
	C := \{ W \in L^2(\Omega_T,\R^n) : \| W \|_{L^\infty(\Omega_T,\R^n)} \leq M \}
$$
is a convex subset of $L^2(\Omega_T,\R^n)$.
Therefore, since $\xi \mapsto f(t,\xi)$ is convex for a.e.~$t \in [0,T]$ and since $\iint_{\Omega_T} f(t,W) \,\dx\dt$ is finite for any $W \in C$ by \eqref{eq:dominating_f}, we obtain that the functional $F \colon L^2(\Omega_T,\R^n) \to (-\infty,\infty]$ given by
$$
	F[W] :=
	\left\{
	\begin{array}{ll}
		\iint_{\Omega_T} f(t,W) \,\dx\dt
		&\text{if } W \in C, \\[5pt]
		\infty
		&\text{else}
	\end{array}
	\right.
$$
is proper and convex.
Further, $F$ is lower semicontinuous with respect to the norm topology in $L^2(\Omega_T,\R^n)$.
Indeed, assume that the sequence $(W_i)_{i \in \N} \subset L^2(\Omega_T,\R^n)$ converges strongly in $L^2(\Omega_T,\R^n)$ to a limit map $W \in L^2(\Omega_T,\R^n)$ as $i \to \infty$.
If $\liminf_{i \to \infty} F[W_i] = \infty$, the assertion $F[W] \leq \liminf_{i \to \infty} F[W_i]$ holds trivially.
Otherwise, there exists a subsequence $\mathfrak{K} \subset \N$ such that $W_i \in C$ for any $i \in \mathfrak{K}$, $\liminf_{i \to \infty} F[W_i] = \lim_{\mathfrak{K} \ni i \to \infty} F[W_i]$
and $W_i \to W$ a.e.~in $\Omega_T$ as $\mathfrak{K} \ni i \to \infty$.
By \eqref{eq:dominating_f} and the dominated convergence theorem, we conclude that $F[W] = \lim_{\mathfrak{K} \ni i \to \infty} F[W_i] = \liminf_{i \to \infty} F[W_i]$.
Therefore, $F$ is also lower semicontinuous with respect to the weak topology in $L^2(\Omega_T,\R^n)$, cf.~\cite[Corollary 2.2]{Ekeland-Temam}.
Since $\Omega_T$ is bounded, we have that $V_i \wto V$ weakly in $L^2(\Omega_T,\R^n)$ as $i \to \infty$ and hence
$$
	\iint_{\Omega_T} f(t,V) \,\dx\dt
	=
	F[V]
	\leq
	\liminf_{i \to \infty} F[V_i]
	=
	\liminf_{i \to \infty} \iint_{\Omega_T} f(t,V_i) \,\dx\dt.
$$
This concludes the proof of the lemma.
\end{proof}

\subsection{Steklov averages of the integrand}
For the final approximation argument in the proof of Theorem \ref{thm:main_existence} we need to regularize the integrand $f$ with respect to time.
To this end, extend $f$ to $[0,\infty] \times \R^n$ by zero if $T<\infty$.
For $\epsilon>0$ define the \emph{Steklov average} $f_\epsilon \colon [0,T] \times \R^n \to \R$ of the extended integrand by
\begin{equation}
	f_\epsilon(t,\xi) := \bint_t^{t+\epsilon} f(s,\xi) \,\ds.
	\label{eq:Steklov}
\end{equation}
In order to investigate convergence of the Steklov averages as $\epsilon \downarrow 0$, first note that specializing the proof of \cite[Corollary 2.4]{Ekeland-Temam} gives us the following result.
\begin{lemma}
\label{lem:quantitative_lipschitz}
Let $L>0$ and assume that $f \colon \R^n \to \R$ is a convex function 
with $\|f\|_{L^\infty(B_{L+1}(0))} \leq C$.
Then, $f$ satisfies the local Lipschitz continuity condition
$$
	|f(\xi_1) - f(\xi_2)| \leq 2C |\xi_1 - \xi_2|
	\quad \text{for all } \xi_1, \xi_2 \in B_L(0).
$$
\end{lemma}
We also need the following variant of the dominated convergence theorem that can be found for example in \cite[Theorem 1.20]{Evans-Gariepy}.
\begin{lemma}
\label{lem:dominated_variant}
Assume that $v, v_k \in L^1(\R^n)$ and $w, w_k \in L^1(\R^n)$ are measurable for all $k \in \N$. Suppose that $w_k \rightarrow w$ a.e.~in $\R^n$ and $|w_k| \leq v_k$  for all $k \in \N$.
Suppose moreover that $v_k \rightarrow v$ a.e.~in $\R^n$ and
\begin{equation*}
	\lim_{k\rightarrow \infty} \int_{\R^n} v_k \,\dx = \int_{\R^n} v \,\dx.
\end{equation*}
Then 
\[
	\lim_{k\rightarrow \infty} \int_{\R^n} |w_k - w| \,\dx = 0.
\]
\end{lemma}

With the preceding lemmas at hand, we prove the following convergence result.
\begin{lemma}
\label{lem:Steklov_convergence}
Let $T \in (0,\infty)$ and assume that $f \colon [0,T] \times \R^n \to \R$ satisfies hypotheses
\eqref{eq:integrand}.
For $\epsilon>0$ let $f_\epsilon \colon [0,T] \times \R^n \to \R$ denote the Steklov average of $f$ given by \eqref{eq:Steklov}.
Then, we have that
$$
	\lim_{\epsilon \downarrow 0}
	\int_0^T \sup_{|\xi| \leq L} |f_\epsilon(t,\xi) - f(t,\xi)| \,\dt
	= 0
	\quad \text{for any } L>0.
$$
\end{lemma}
\begin{proof}
Fix $L>0$.
First of all, we show that
\begin{equation}
	\lim_{\epsilon \downarrow 0} \sup_{|\xi| \leq L} |f_\epsilon(t,\xi) - f(t,\xi)|
	= 0
	\quad \text{for a.e.~$t \in [0,T]$.}
	\label{eq:steklov_aux1}
\end{equation}
By \eqref{eq:integrand}$_2$, for fixed $\xi \in \R^n$ we have that $f_\epsilon(t,\xi) \to f(t,\xi)$ for a.e.~$t \in [0,T]$ by Lebesgue's differentiation theorem.
Thus, there exists a set $N$ of $\L^1$-measure zero such that
\begin{equation}
	f_\epsilon(t,\xi) \to f(t,\xi)
	\quad \text{for any $t \in [0,T] \setminus N$ and $\xi \in \mathds{Q}^n$.}
	\label{eq:steklov_aux2}
\end{equation}
Without loss of generality assume that additionally for all $t \in [0,T] \setminus N$ the map $\xi \mapsto f(t,\xi)$ is convex, the function $g_{L+1}$ from \eqref{eq:dominating_f} fulfills $g_{L+1}(t) < \infty$ and there holds $\bint_t^{t+\epsilon} g_{L+1}(s) \,\ds \to g_{L+1}(t)$.
Now, fix $t \in [0,T] \setminus N$.
By \eqref{eq:dominating_f} and Lemma~\ref{lem:quantitative_lipschitz} we conclude that $\xi \mapsto f(t,\xi)$ is Lipschitz continuous in $B_L(0)$
with Lipschitz constant $2 g_{L+1}(t)$.
Using this together with the definition of the Steklov average, for any $\xi_1, \xi_2 \in B_L(0)$ we compute that
\begin{align*}
	|f_\epsilon(t,\xi_1) - f_\epsilon(t,\xi_2)|
	\leq
	\bint_t^{t+\epsilon} |f(s,\xi_1) - f(s,\xi_2)| \,\ds
	\leq
	2 \bint_t^{t+\epsilon} g_{L+1}(s) \,\ds \,|\xi_1 - \xi_2|.
\end{align*}
Since $\bint_t^{t+\epsilon} g_{L+1}(s) \,\ds \to g_{L+1}(t)$, there exists $\epsilon_o>0$ such that $\xi \mapsto f_\epsilon(t,\xi)$ is Lipschitz continuous with Lipschitz constant $4 g_{L+1}(t)$ for all $\epsilon \in (0,\epsilon_o]$.
This shows that the sequence $(f_\epsilon(t,\cdot))_{\epsilon \in (0,\epsilon_o]}$ is equicontinuous in $B_L(0)$.
Moreover, $(f_\epsilon(t,\cdot))_{\epsilon \in (0,\epsilon_o]}$ is equibounded in $B_L(0)$, since for any $\xi \in B_L(0)$ and $\epsilon \in (0,\epsilon_o]$, we find that
$$
	|f_\epsilon(t,\xi)|
	\leq
	\bint_t^{t+\epsilon} |f(s,\xi)| \,\ds
	\leq
	\bint_t^{t+\epsilon} g_{L+1}(s) \,\ds
	\leq
	2 g_{L+1}(t).
$$
Therefore, we infer from the Arz\`ela-Ascoli theorem that $(f_\epsilon(t,\cdot))_{\epsilon \in (0,\epsilon_o]}$ converges uniformly in $B_L(0)$ as $\epsilon \downarrow 0$ and the limit $f(t,\cdot)$ is determined by \eqref{eq:steklov_aux2}.
This concludes the proof of \eqref{eq:steklov_aux1}.
Next, since
$$
	\sup_{|\xi| \leq L} |f_\epsilon(t,\xi) - f(t,\xi)|
	\leq
	\sup_{|\xi| \leq L} |f_\epsilon(t,\xi)| + \sup_{|\xi| \leq L} |f(t,\xi)|
	\leq
	\bint_t^{t+\epsilon} g_L(s) \,\ds + g_L(t),
$$
where $\bint_t^{t+\epsilon} g_L(s) \,\ds \to g_L(t)$ in $L^1(0,T)$, the claim now follows from Lemma \ref{lem:dominated_variant}.
\end{proof}

\subsection{Mollification in time}
In general, variational solutions are not admissible as comparison maps in the variational inequality \eqref{eq:var_ineq_unconstrained}, since they do not necessarily admit a derivative with respect to time.
Therefore, we use the following mollification procedure with respect to time.
More precisely, consider a separable Banach space $X$, an initial datum $v_o \in X$ and a map $v \in L^r(0,T;X)$ for some $r \in [1,\infty]$.
For $h>0$ define the mollification
\begin{equation}
	[v]_h(t) :=
	e^{-\frac{t}{h}} v_o + \tfrac1h \int_0^t e^\frac{s-t}{h} v(s) \,\ds
	\quad\text{for any } t \in [0,T].
	\label{eq:time_mollification}
\end{equation}
Later on, we will mainly use $X= L^q(\Omega)$ or $X = W^{1,q}(\Omega)$ for some $q \in [1,\infty)$.
A vital feature of this mollification procedure is that $[v]_h$ solves the ordinary differential equation
\begin{equation}
	\partial_t [v]_h = \tfrac1h \big( v - [v]_h \big)
	\label{eq:ode_mollification}
\end{equation}
with initial condition $[v]_h(0) = v_o$.
This shows in particular that if $v$ and $[v]_h$ are contained in a function space, the same holds true for the time derivative of $[v]_h$.
The basic properties of time mollifications are collected in the following lemma (cf.~\cite[Lemma 2.2]{Kinnunen-Lindqvist} and \cite[Appendix B]{BDM} for the proofs).
\begin{lemma}
\label{lem:time_mollification}
Let $X$ be a separable Banach space and $v_o \in X$.
If $v \in L^r(0,T;X)$ for some $r \in [1,\infty]$, then also $[v]_h \in L^r(0,T;X)$
and if $r < \infty$, then $[v]_h \to v$ in $L^r(0,T;X)$ as $h \downarrow 0$.
Further, for any $t_o \in (0,T]$ there holds the bound
$$
	\big\| [v]_h \big\|_{L^r(0,t_o;X)}
	\leq
	\| v \|_{L^r(0,t_o;X)}
	+ \Big[ \tfrac{h}{r} \Big( 1 - e^{-\tfrac{t_o r}{h}} \Big) \Big]^\frac{1}{r} \|v_o\|_X,
$$
where the bracket $[\ldots]^\frac{1}{r}$ has to be interpreted as $1$ if $r = \infty$.
Moreover, if $v \in C^0([0,T];X)$, then also $[v]_h \in C^0([0,T];X)$ with $[v]_h(0) = v_o$ and there holds $[v]_h \to v$ in $L^\infty(0,T;X)$ as $h \downarrow 0$.
\end{lemma}

For maps $v \in L^r(0,T;X)$ with $\partial_t v \in L^r(0,T;X)$
we have the following assertion.
\begin{lemma}
\label{lem:time_mollification_2}
Let $X$ be a separable Banach space and $r \ge 1$.
Assume that $v \in L^r(0,T;X)$ with $\partial_t v \in L^r(0,T;X)$.
Then, for the mollification in time defined by
\begin{equation*}
	[v]_h(t)
	:=
	e^{-\frac{t}{h}}v(0) + \tfrac{1}{h} \int_0^t e^\frac{s-t}{h}v(s) \, \ds
\end{equation*}
the time derivative can be computed by
\begin{equation*}
	\partial_t [v]_h(t)
	=
	\tfrac{1}{h} \int_0^t e^\frac{s-t}{h}\partial_s v(s) \, \ds.
\end{equation*}
\end{lemma}

\section{Properties of variational solutions}\label{sec:properties}
As mentioned in the introduction, besides variational solutions in the sense of Definition~\ref{def:unconstrainded_solution}, we consider variational solutions of the so-called gradient constrained obstacle problem to \eqref{eq:pde}. They enjoy the same basic properties as variational solutions to the unconstrained Cauchy-Dirichlet problem to \eqref{eq:pde} and proofs will be given in a unified way in this section.  

Let  $L \in (0, \infty]$. We define the following class of functions that are $L$-Lipschitz in space
$$
	K^L(\Omega_T) :=
	\{v \in K^\infty(\Omega_T) : \|Dv\|_{L^\infty(\Omega_T,\R^n)} \leq L \}
$$
and given time-independent boundary values $u_o \in W^{1,\infty}(\Omega)$ with  $\|Du_o\|_{L^\infty(\Omega,\R^n)} \leq L$, we denote the subclass 
$$
	K^L_{u_o}(\Omega_T) :=
	\{v \in K^L(\Omega_T) :
	\text{$v = u_o$ on the lateral boundary $\partial\Omega \times (0,T)$} \}.
$$

\begin{definition}
\label{def:constrained_solution}
Assume that $f \colon [0,T] \times \R^n \to \R$ satisfies \eqref{eq:integrand}, consider a boundary datum $u_o \in W^{1,\infty}(\Omega)$
and let $L \in (0,\infty)$ be such that $\|Du_o\|_{L^\infty(\Omega,\R^n)} \leq L$.
In the case $T<\infty$ a map $u \in K^L_{u_o}(\Omega_T)$ is called a \emph{variational solution}
to the gradient constrained Cauchy-Dirichlet problem associated with \eqref{eq:pde} and $u_o$ in $\Omega_T$
if and only if the variational inequality
\begin{align}
	\iint_{\Omega_T} f(t,Du) \,\dx\dt
	&\leq
	\iint_{\Omega_T} \partial_t v (v - u) + f(t,Dv) \,\dx\dt \label{eq:var_ineq_constrained} \\
	&\phantom{=}
	+ \tfrac12 \|v(0) - u_o\|_{L^2(\Omega)}^2
	- \tfrac12 \|(v - u)(T)\|_{L^2(\Omega)}^2
	\nonumber
\end{align}
holds true for any comparison map $v \in K^L_{u_o}(\Omega_T)$
with $\partial_t v \in L^2(\Omega_T)$.
If $T=\infty$ and $u \in K^L_{u_o}(\Omega_\infty)$ is a variational solution
in $\Omega_\tau$ for any $\tau > 0$, $u$ is called a \emph{global variational solution}
or variational solution in $\Omega_\infty$ to the gradient constrained Cauchy-Dirichlet problem
associated with \eqref{eq:pde} and $u_o$.
\end{definition}

\subsection{Continuity with respect to time}
In Definitions \ref{def:unconstrainded_solution} and \ref{def:constrained_solution} we require that variational solutions are contained in the space $C^0([0,T];L^2(\Omega))$.
However, this is already implied if $u$ satisfies a variational inequality for a.e.~$\tau \in [0,T]$.
More precisely, we have the following Lemma, which will be applied with $L=\infty$ in Section \ref{sec:unconstrained_for_general}.
\begin{lemma}
\label{lem:time_continuity}
Let $\Omega \subset \R^n$ be open and bounded and $T\in(0,\infty)$ and assume that $f \colon [0,T] \times \R^n \to \R$ satisfies \eqref{eq:integrand}.
Let $L \in (0,\infty]$ and consider $u_o\in W^{1,\infty}(\Omega)$ such that $\|Du_o\|_{L^\infty(\Omega,\R^n)} \leq L$.
Further, consider $u\in L^{\infty}(\Omega_{T})$ with $u=u_o$ on $\partial_{\mathcal{P}}\Omega_{T}$ and $\left\Vert Du\right\Vert _{L^{\infty}(\Omega_T,\mathbb{R}^{n})} \leq L$ if $L \in (0,\infty)$ and $\left\Vert Du\right\Vert _{L^{\infty}(\Omega_T,\mathbb{R}^{n})}<\infty$ if $L=\infty$, respectively.
Suppose that $u$ satisfies the variational inequality 
\begin{align}
	\iint_{\Omega_{\tau}}f(t,Du)\,\dx\dt & \leq\iint_{\Omega_{\tau}}\partial_{t}v(v-u)\,\dx\dt+\iint_{\Omega_{\tau}}f(t,Dv)\,\dx\dt\nonumber \\
	 &\phantom{=} +\tfrac{1}{2}\left\Vert v(0)-u_o\right\Vert _{L^{2}(\Omega)}^{2}-\tfrac{1}{2}\left\Vert v(\tau)-u(\tau)\right\Vert _{L^{2}(\Omega)}^{2}\label{eq:time continuity 1}
\end{align}
for almost all $\tau\in(0,T)$ whenever $v\in K_{u_o}^L(\Omega_{T})$
with $\partial_{t}v\in L^{2}(\Omega_{T})$.
Then, we have that $u\in C^{0}([0,T];L^{2}(\Omega))$.
\end{lemma}

\begin{proof}
The proof is similar to that of Lemma 2.6 in \cite{Schaetzler17} except for the
estimate of the second integral in \eqref{eq:time continuity 2} below. Denote by $[u]_h$ the time mollification of $u$ with initial values $u_o$ as defined in \eqref{eq:time_mollification}.
In particular, observe that $[u]_{h}\in C^{0}([0,T];L^{2}(\Omega))$, since we know that $\partial_{t}[u]_{h}\in L^{2}(\Omega_{T})$ and $[u]_{h}(0) = u_o \in L^{2}(\Omega)$.
Using $[u]_h$ as a comparison function in \eqref{eq:time continuity 1}, taking the essential supremum over $\tau\in(0,T)$ and recalling that $([u]_h - u) = - h \partial_t [u]_h$, we obtain that
\begin{align}
	\sup_{\tau\in(0,T)}\tfrac{1}{2}\left\Vert [u]_h(\tau)-u(\tau)\right\Vert _{L^{2}(\Omega)}^{2} & \leq\sup_{\tau\in(0,T)}\iint_{\Omega_{\tau}}\partial_{t}  [u]_{h}([u]_{h}-u)\,\dx\dt\nonumber\\
	&\phantom{=}+\iint_{\Omega_{T}}f(t,D[u]_{h})-f(t,Du)\,\dx\dt\nonumber\\
	&\leq \iint_{\Omega_{T}}|f(t,D[u]_{h})-f(t,Du)|\,\dx\dt.
	\label{eq:time continuity 2}
\end{align}
Furthermore, we have that $D[u]_{h}\rightarrow Du$ almost everywhere in
$\Omega_{T}$ as $h\downarrow0$ (up to a subsequence) and that $|D[u]_{h}| \leq |Du_o| + \sup_{\Omega_T} |Du|$. Therefore, by \eqref{eq:dominating_f} and the dominated
convergence theorem we find that the second integral in \eqref{eq:time continuity 2}
vanishes in the limit $h\downarrow0$.
Hence, we have shown that $[u]_{h}\rightarrow u$ in $L^{\infty}(0,T;L^2(\Omega))$.
Combining this with the fact that $[u]_h \in C^0([0,T];L^2(\Omega))$, it follows that also $u\in C^{0}([0,T];L^2(\Omega))$. 
\end{proof}

\subsection{Localization in time}
Here, we show that a variational solution in a space-time cylinder $\Omega_T$ is also a solution in any sub-cylinder $\Omega_\tau$, $\tau \in (0,T)$.
\begin{lemma}[Localization in time]\label{lem:local time}
Let $T \in (0,\infty)$, assume that $\Omega \subset \mathbb{R}^n$ is open and bounded, and that $f \colon [0,T] \times \R^n \to \R$ satisfies \eqref{eq:integrand}.
Consider $u_o \in W^{1,\infty}(\Omega)$ and $L \in (0, \infty]$ such that $\|Du_o\|_{L^\infty(\Omega,\R^n)} \leq L$.
Suppose that $u$ is a variational solution to \eqref{eq:pde} in $K^{L}_{u_o}(\Omega_{T})$ (in the sense of Definition \ref{def:constrained_solution} if $L < \infty$, in the sense of Definition \ref{def:unconstrainded_solution} if $L = \infty$).
Then $\left.u\right|_{\Omega_{\tau}}$ is a variational solution to \eqref{eq:pde} in $K^{L}_{u_o}(\Omega_{\tau})$ for any $\tau\in(0,T]$. 
\end{lemma}
\begin{proof}
For $\theta\in(0,\tau)$, consider the cut-off function
\[
	\xi_{\theta}(t):=\chi_{[0,\tau-\theta]}(t)+\frac{\tau-t}{\theta}\chi_{(\tau-\theta,\tau]}(t).
\]
For $v\in K_{u_o}^{L}(\Omega_{\tau})$ satisfying $\partial_{t}v\in L^{2}(\Omega_{\tau})$
we define a function $v_{\theta} \colon \Omega_{T}\rightarrow\mathbb{R}$
by 
\[
	v_{\theta}:=\xi_{\theta}v+(1-\xi_{\theta})[u]_{h},
\]
where $\xi_{\theta}v$ has been extended to $\Omega_{T}$ by zero and $[u]_h$ is defined according to \eqref{eq:time_mollification} with initial datum $u_o$. Then
we have $v_{\theta}\in K_{u_o}^{L}(\Omega_{T})$ with $\partial_{t}v_{\theta}\in L^{2}(\Omega_{T})$,
and therefore we may use $v_{\theta}$ as a comparison map for $u$
in the variational inequality. This yields
\begin{align}
	\iint_{\Omega_{T}}f(t,Du)\,\dx\dt & \leq \iint_{\Omega_{T}}\partial_{t}v_{\theta}(v_{\theta}-u)+f(t,Dv_{\theta})\,\dx\dt\nonumber \\
	 &\phantom{=}+
	 \tfrac{1}{2} \left\Vert v(0)-u_o\right\Vert _{L^{2}(\Omega)}^{2}
	 -\tfrac{1}{2} \left\Vert ([u]_{h}-u)(T)\right\Vert _{L^{2}(\Omega)}^{2}. \label{eq:local time 1}
\end{align}
The first term on the right-hand side of (\ref{eq:local time 1}) is identical
to the one in \cite[Equation (3.2)]{BDMS}
and can be estimated in the same way to obtain
\begin{align*}
	& \limsup_{\theta\rightarrow0}\iint_{\Omega_{T}}\partial_{t}v_{\theta}(v_{\theta}-u)\,\dx\dt \\
	& \phantom{=} \leq \iint_{\Omega_{\tau}}\partial_{t}v(v-u)\,\dx\dt 
	+\iint_{\Omega\times(\tau,T)}\partial_{t}[u]_{h}([u]_{h}-u)\,\dx\dt
	-\tfrac{1}{2}\int_{\Omega}(v-[u]_{h})^{2}(\tau)\,\dx\\
	& \phantom{=}\phantom{leq}+\int_{\Omega}([u]_{h}-u)(v-[u]_{h})(\tau)\,\dx.
\end{align*}
The second term on the right-hand side of (\ref{eq:local time 1})
is given by
\begin{align*}
\iint_{\Omega_{T}}f(t,Dv_{\theta})\,\dx\dt & =  \iint_{\Omega\times(\tau-\theta,\tau)}f(t,\xi_{\theta}Dv+(1-\xi_{\theta})D[u]_{h})\,\dx\dt\\
 &\phantom{=}+\iint_{\Omega\times(0,\tau-\theta)}f(t,Dv)\,\dx\dt+\iint_{\Omega\times(\tau,T)}f(t,D[u]_{h})\,\dx\dt.
\end{align*}
Since we know that
\begin{align*}
	\|\xi_{\theta}Dv&+(1-\xi_{\theta})D[u]_{h}\|_{L^\infty(\Omega_T,\R^n)}
	\leq
	\|Dv\|_{L^\infty(\Omega_T,\R^n)} + \|D[u]_h\|_{L^\infty(\Omega_T,\R^n)} \\
	&\leq
	\|Dv\|_{L^\infty(\Omega_T,\R^n)} + \|Du_o\|_{L^\infty(\Omega,\R^n)} + \|Du\|_{L^\infty(\Omega_T,\R^n)}
	=: M < \infty,
\end{align*}
by \eqref{eq:dominating_f} we find that
$$
	\bigg| \iint_{\Omega\times(\tau-\theta,\tau)} f(t,\xi_{\theta}Dv+(1-\xi_{\theta})D[u]_{h})\,\dx\dt \bigg|
	\leq
	|\Omega| \int_{\tau-\theta}^\tau g_M(t) \,\dt
	\to 0
$$
in the limit $\theta \downarrow 0$.
Combining the preceding estimates we arrive at
\begin{align}
	\iint_{\Omega_{T}}f(t,Du)\,\dx\dt \nonumber
	&\leq
	\iint_{\Omega\times(0,\tau)}f(t,Dv)\,\dx\dt
	+\iint_{\Omega\times(\tau,T)}f(t,D[u]_{h})\,\dx\dt\nonumber\\
	&\phantom{==} -\tfrac{1}{2}\int_{\Omega}(v-[u]_{h})^{2}(\tau)\,\dx+\int_{\Omega}([u]_{h}-u)(v-[u]_{h})(\tau)\,\dx\nonumber\\
	&\phantom{==} +\iint_{\Omega_{\tau}}\partial_{t}v(v-u)\,\dx\dt
	+\iint_{\Omega\times(\tau,T)}\partial_{t}[u]_{h}([u]_{h}-u)\,\dx\dt\nonumber\\
	&\phantom{==} +\tfrac{1}{2} \left\Vert v(0)- u_o\right\Vert _{L^{2}(\Omega)}^{2}
	-\tfrac{1}{2} \left\Vert ([u]_{h}-u)(T)\right\Vert _{L^{2}(\Omega)}^{2}. \label{eq:local time 2}
\end{align}
Note that $[u]_h \to u$ in $L^\infty(0,T;L^2(\Omega))$ as $h \downarrow 0$, since $u \in C^0([0,T];L^2(\Omega))$.
Further, we have that $D[u]_{h}\rightarrow Du$ pointwise almost everywhere in $\Omega_{T}$ as $h \downarrow 0$ (up to a subsequence) and that
$$
	\left\Vert D[u]_{h}\right\Vert _{L^{\infty}(\Omega_{T},\R^n)}
	\leq
	\|Du_o\|_{L^\infty(\Omega,\R^n)}
	+ \left\Vert Du\right\Vert _{L^{\infty}(\Omega_{T},\R^n)}
	=: L' <\infty
	\quad \text{for any } h>0.
$$
Therefore, assumption \eqref{eq:dominating_f}, the fact that $\Omega$ is bounded and the dominated convergence theorem imply that
\[
\lim_{h\downarrow0} \iint_{\Omega\times(\tau,T)} f(t,D[u]_h) \,\dx\dt = \iint_{\Omega\times(\tau, T)} f(t, Du) \,\dx\dt.
\]
Hence, using that $\partial_{t}[u]_{h}([u]_{h}-u)\leq0$ and letting $h\downarrow0$ in \eqref{eq:local time 2}, we obtain the desired inequality
\begin{align*}
	\iint_{\Omega_{\tau}} f(t,Du)\,\dx\dt
	& \leq
	\iint_{\Omega_{\tau}} \partial_{t}v(v-u) + f(t,Dv)\,\dx\dt\\
 	&\phantom{=}
 	+\tfrac{1}{2} \left\Vert v(0)-u_o\right\Vert _{L^{2}(\Omega)}^{2}
 	-\tfrac{1}{2} \left\Vert (v-u)(\tau)\right\Vert _{L^{2}(\Omega)}^{2} \qedhere.
\end{align*}
\end{proof}

\subsection{The initial condition}
As a consequence of the localization in time principle, we find that variational solutions attain the initial datum $u_o$ in the $C^0$-$L^2$-sense. 
The precise statement is as follows.

\begin{lemma}
\label{lem:initial_datum}
Let $T \in (0,\infty)$, assume that $\Omega\subset\mathbb{R}^{n}$ is bounded and open, and that $f \colon [0,T]\times\mathbb{R}^{n}\rightarrow\mathbb{R}$ satisfies \eqref{eq:integrand}. Consider $u_o\in W^{1,\infty}(\Omega)$ and $L\in(0,\infty]$ such that $\left\Vert Du_o\right\Vert _{L^{\infty}(\Omega,\mathbb{R}^{n})}\leq L$. Suppose that $u$ is a variational solution to \eqref{eq:pde} in $K_{u_o}^{L}(\Omega_{T})$  (in the sense of Definition \ref{def:constrained_solution} if $L<\infty$, in the sense of Definition \ref{def:unconstrainded_solution} if $L=\infty$). Then, there holds
\[
	\lim_{\tau\downarrow0}\left\Vert u(\tau)-u_o\right\Vert _{L^{2}(\Omega)}^{2}=0.
\]
\end{lemma}
\begin{proof}
By Lemma \ref{lem:local time}, the function $u$ is a variational solution in any smaller cylinder
$\Omega_{\tau}$, $\tau \in (0,T]$. Using $v \colon \Omega_{\tau}\rightarrow\mathbb{R}$,
$v(x,t):=u_o(x)$ as a comparison function for $u$ and taking \eqref{eq:dominating_f} with $M:= \max\{ \|Du\|_{L^{\infty}(\Omega_{T},\mathbb{R}^{n})}, \|Du_o\|_{L^{\infty}(\Omega,\mathbb{R}^{n})} \}$ into account, we obtain that
\[
\tfrac{1}{2}\left\Vert u(\tau)-u_o\right\Vert _{L^{2}(\Omega)}\leq\iint_{\Omega_{\tau}}f(t,Du_o)-f(t,Du)\,\dx\dt\leq2\left|\Omega\right|\int_{0}^{\tau}g_{M}(t)\,\dt.
\]
Since $g_M \in L^1(0,T)$, this implies the claim.
\end{proof}

\subsection{Comparison principle}
The following comparison principle ensures in particular that variational solutions to the problems considered in the present paper are unique.
\begin{theorem}[Comparison principle]
\label{thm:comparison principle} Let $T \in (0,\infty)$, assume that $\Omega \subset \mathbb{R}^n$ is bounded and open, and that $f \colon [0,T] \times \R^n \to \R$ satisfies \eqref{eq:integrand}. Let $L \in (0, \infty]$ and suppose that $u$ and $\tilde{u}$ are variational
solutions to \eqref{eq:pde} in $K^{L}(\Omega_{T})$ (in the sense of Definition \ref{def:constrained_solution} if $L < \infty$ and in the sense of Definition \ref{def:unconstrainded_solution} if $L = \infty$) such that
$\|Du(0)\|_{L^{\infty}(\Omega,\R^n)}$ and $\|D\tilde{u}(0)\|_{L^{\infty}(\Omega,\R^n)}$ are bounded by $L$ if $L \in (0,\infty)$ and finite if $L=\infty$, respectively.
Then the assumption that
\begin{equation*}
u\leq\tilde{u} \quad \text{on} \quad \partial_{\mathcal{P}}\Omega_{T}
\end{equation*}
implies
\begin{equation*}
u\le\tilde{u} \quad \text{in} \quad \Omega_{T}.
\end{equation*}
\end{theorem}

\begin{proof}
Let $\tau\in(0,T]$. By Lemma \ref{lem:local time}, the functions $u$ and $\tilde{u}$
are variational solutions in $K^{L}(\Omega_{\tau})$. Consider the
functions
\[
v:=\min([u]_{h},[\tilde{u}]_{h})\quad\text{and}\quad w:=\max([u]_{h},[\tilde{u}]_{h}),
\]
where $[u]_h$ and $[\tilde{u}]_h$ denote the mollifications of $u$ and $\tilde{u}$ according to \eqref{eq:time_mollification} with initial values $u(0) \in W^{1,\infty}(\Omega)$ and $\tilde{u}(0) \in W^{1,\infty}(\Omega)$, respectively.
Since the boundary values attained by $u$ and $\tilde{u}$ are independent of time,
we have that $v\in K_{u}^{L}(\Omega_{\tau})$ and $w\in K_{\tilde{u}}^{L}(\Omega_{\tau})$
with $\partial_{t}v,\partial_{t}w\in L^{2}(\Omega_{\tau})$. Therefore
we may use $v$ and $w$ as comparison functions in the variational
inequalities of $u$ and $\tilde{u}$, respectively. Adding the resulting
inequalities and using that $[u]_h(0) = u(0) \leq \tilde{u}(0) = [\tilde{u}]_h(0)$, we obtain
\begin{align}
	0\leq &
	\iint_{\Omega_{\tau}}\partial_{t}v(v-u)+\partial_{t}w(w-\tilde{u})\, \dx \dt \nonumber\\
 	& +\iint_{\Omega_{\tau}}f(t,Dv)-f(t,Du)+f(t,Dw)-f(t,D\tilde{u})\, \dx \dt \nonumber \\
 	& -\tfrac{1}{2}\left\Vert (v-u)(\tau)\right\Vert _{L^{2}(\Omega)}^{2}-\tfrac{1}{2}\left\Vert (w-\tilde{u})(\tau)\right\Vert _{L^{2}(\Omega)}^{2}. 	\label{eq:comparison 1}
\end{align}
Using the identities
\begin{align*}
	\left\{
	\begin{array}{l}
		v-u=\min([u]_{h},[\tilde{u}]_{h})-[u]_{h}-(u-[u]_{h}) =-([u]_{h}-[\tilde{u}]_{h})_{+}-h\partial_{t}[u]_{h}, \\[5pt]
		w-\tilde{u} = ([u]_{h}-[\tilde{u}]_{h})_{+} - h\partial_{t}[\tilde{u}]_{h},
	\end{array}
	\right.
\end{align*}
we compute that
\begin{align*}
	&\partial_{t}v (v-u)+\partial_{t}w(w-\tilde{u})\\
	&=
	\big(\partial_{t}[u]_{h} \chi_{\{ [u]_{h}\leq[\tilde{u}]_{h}\} }
	+\partial_{t}[\tilde{u}]_{h}\chi_{\{ [\tilde{u}]_{h}<[u]_{h}\} }\big)
	\big(- \big([u]_{h}-[\tilde{u}]_{h} \big)_{+}-h\partial_{t}[u]_{h}\big)\\
	&\phantom{=}
	+\big(\partial_{t}[\tilde{u}]_{h} \chi_{\{ [u]_{h}\leq[\tilde{u}]_{h}\} }
	+\partial_{t}[u]_{h} \chi_{\{ [\tilde{u}]_{h}<[u]_{h}\} }\big)
	\big( \big([u]_{h}-[\tilde{u}]_{h} \big)_{+}-h\partial_{t}[\tilde{u}]_{h}\big)\\
	&=
	\big( \partial_{t}[\tilde{u}]_{h} \big( [u]_{h}-[\tilde{u}]_{h} \big)_{+}
	-\partial_{t}[u]_{h} ([u]_{h} - [\tilde{u}]_{h})_{+}
	- h(\partial_{t}[u]_{h})^{2}
	- h(\partial_{t}[\tilde{u}]_{h})^{2}\big)\\
	&\phantom{=}
	\cdot\chi_{\{ [u]_{h}\leq[\tilde{u}]_{h}\} }\\
	&\phantom{=}
	+\big( \partial_{t}[u]_{h} \big( [u]_{h}-[\tilde{u}]_{h} \big)_{+}
	- \partial_{t}[\tilde{u}]_{h} \big( [u]_{h}-[\tilde{u}]_{h} \big)_{+}
	- h\partial_{t}[\tilde{u}]_{h}\partial_{t}[u]_{h}
	- h\partial_{t}[u]_{h}\partial_{t}[\tilde{u}]_{h}\big)\\
	&\phantom{=}
	\cdot\chi_{\{ [\tilde{u}]_{h}<[u]_{h}\} }\\
	&\leq
	\big( \partial_{t}[u]_{h} \big( [u]_{h}-[\tilde{u}]_{h} \big)_{+}
	-\partial_{t}[\tilde{u}]_{h} \big( [u]_{h}-[\tilde{u}]_{h} \big)_{+}
	- h\partial_{t}[\tilde{u}]_{h} \partial_{t}[u]_{h}
	- h\partial_{t}[u]_{h}\partial_{t}[\tilde{u}]_{h} \big)\\
	&\phantom{=}
	\cdot\chi_{\{ [\tilde{u}]_{h}<[u]_{h} \} }\\
	&=
	\partial_{t} \big( [u]_{h}-[\tilde{u}]_{h} \big) \big([u]_{h}-[\tilde{u}]_{h} \big)_{+}
	-2h \partial_{t}[u]_{h} \partial_{t}[\tilde{u}]_{h}
	\chi_{\{ [\tilde{u}]_{h}<[u]_{h} \} }\\
	&\leq
	\tfrac{1}{2} \partial_{t} \big( \big([u]_{h}-[\tilde{u}]_{h} \big)_{+})^{2} \big)
	+ h \big( \big(\partial_{t}[u]_{h} \big)^{2} + \big( \partial_{t}[\tilde{u}]_{h} \big)^{2} \big).
\end{align*}
Therefore, taking into account that $[u]_h(0) = u(0) \leq \tilde{u}(0) = [\tilde{u}]_h(0)$,
we find that
\begin{align}
	\iint_{\Omega_{\tau}} &\partial_{t}v(v-u)+\partial_{t}w(w-\tilde{u})\, \dx \dt \nonumber \\
	& \leq
 	\tfrac{1}{2} \big\Vert \big([u]_{h}-[\tilde{u}]_{h} \big)_{+}(\tau) \big\Vert _{L^{2}(\Omega)}^{2}
	+\iint_{\Omega_{\tau}} h \big( \big(\partial_{t}[u]_{h} \big)^{2} + \big( \partial_{t}[\tilde{u}]_{h} \big)^{2} \big)\, \dx \dt.
	\label{eq:comparison 2}
\end{align}
Furthermore, using $[u]_{h}$ as a comparison function for $u$ and omitting the boundary term at time $\tau$ on the right-hand side of the variational inequality, we obtain
\begin{align}
	\iint_{\Omega_{\tau}} h \big( \partial_{t}[u]_{h} \big)^{2}\, \dx \dt
	&=
	- \iint_{\Omega_{\tau}} \partial_{t}[u]_{h} \big( [u]_{h} - u \big) \, \dx \dt \nonumber \\
 	& \leq
 	\iint_{\Omega_{\tau}} f \big( t, D[u]_{h} \big) - f(t,Du)\, \dx \dt\label{eq:comparison 3}
\end{align}
and a similar inequality holds for $\tilde{u}.$ Observe also that
\begin{align}
 	f(t,Dv) &- f(t,Du)+f(t,Dw)-f(t,D\tilde{u})\nonumber \\
 	& =
 	\chi_{\{ [u]_{h}\leq[\tilde{u}]_{h} \} } f\big(t, D[u]_{h} \big)
 	+\chi_{\{ [\tilde{u}]_{h}<[u]_{h} \} } f\big(t, D[\tilde{u}]_{h} \big)
 	-f(t,Du) \nonumber \\
	&\phantom{=}
	+\chi_{\{ [u]_{h}\leq[\tilde{u}]_{h}\} } f\big(t, D[\tilde{u}]_{h} \big)
	+\chi_{\{ [\tilde{u}]_{h}<[u]_{h}\} } f\big(t, D[u]_{h} \big)
	-f(t, D\tilde{u})\nonumber \\
 	& =
 	f\big(t,D[u]_{h}\big)-f(t,Du)+f\big(t,D[\tilde{u}]_{h}\big)-f(t,D\tilde{u}).\label{eq:comparison 4}
\end{align}
Combining the estimates \eqref{eq:comparison 2}, \eqref{eq:comparison 3}
and \eqref{eq:comparison 4} with \eqref{eq:comparison 1} we arrive
at
\begin{align}
	-\tfrac{1}{2} \big\Vert &\big([u]_{h}-[\tilde{u}]_{h}\big)_{+}(\tau)\big\Vert _{L^{2}(\Omega)}^{2}
 	+\tfrac{1}{2} \Vert (v-u)(\tau) \Vert _{L^{2}(\Omega)}^{2}
 	+\tfrac{1}{2} \Vert (w-\tilde{u})(\tau) \Vert _{L^{2}(\Omega)}^{2} \nonumber \\
	&\leq
 	2\iint_{\Omega_{\tau}} f\big(t,D[u]_{h}\big)-f(t,Du)+f\big(t,D[\tilde{u}]_{h}\big)-f(t,D\tilde{u})\, \dx \dt.
 	\label{eq:comparison 5}
\end{align}
By the same argument as in the end of proof of Lemma \ref{lem:local time} involving the dominated convergence theorem, the integral on the right-hand side of \eqref{eq:comparison 5} vanishes in the limit $h \downarrow 0$.
Writing $v-u = -([u]_{h}-[\tilde{u}]_{h})_{+} + [u]_h - u$ and
$w - \tilde{u} = ([u]_{h}-[\tilde{u}]_{h})_{+} + [\tilde{u}]_h - \tilde{u}$
and using that $[u]_{h}\rightarrow u$ and $[\tilde{u}]_h \to \tilde{u}$
in $L^{\infty}([0,\tau],L^{2}(\Omega))$ as $h \downarrow 0$ since $u, \tilde{u} \in C^0([0,T];L^2(\Omega))$, we conclude that
\begin{align*}
	\lim_{h \downarrow 0}
	&\Big(
	-\tfrac{1}{2} \big\Vert \big([u]_{h}-[\tilde{u}]_{h}\big)_{+}(\tau)\big\Vert _{L^{2}(\Omega)}^{2}
 	+\tfrac{1}{2} \Vert (v-u)(\tau) \Vert _{L^{2}(\Omega)}^{2}
 	+\tfrac{1}{2} \Vert (w-\tilde{u})(\tau) \Vert _{L^{2}(\Omega)}^2 \Big)\\
 	&=
 	\tfrac12 \big\Vert (u-\tilde{u})_{+}(\tau)\big\Vert _{L^{2}(\Omega)}^{2}.
\end{align*}
Hence, taking the limit $h \downarrow 0$ in \eqref{eq:comparison 5}, we infer
$$
	\tfrac12 \big\Vert (u-\tilde{u})_{+}(\tau)\big\Vert _{L^{2}(\Omega)}^{2}
	\leq 0,
$$
which implies that $u\leq\tilde{u}$ in $\Omega_{\tau}$. Since $\tau$
was arbitrary, the claim follows.
\end{proof}

\subsection{Maximum principle and localization in space for regular solutions}
In this section, we consider more regular variational solutions $u$ satisfying $\partial_{t}u \in L^{2}(\Omega_{T})$.
As a consequence, $u$ is directly admissible as comparison map in its variational inequality without regularization with respect to the time variable.
Further, due to the requirements of the proof of the existence result in Section~\ref{sec:unconstrained_bsc}, we will take time-dependent boundary values $\left. u \right|_{\Omega \times (0,T)}$ into account here.
In particular, the proof of the comparison principle in Theorem~\ref{thm:comparison principle} is easily adapted to allow time-dependent boundary values if $\partial_{t}u$ and $\partial_t \tilde{u}$ are contained in $L^{2}(\Omega_{T})$ by using $\min(u,\tilde{u})$ and $\max(u,\tilde{u})$ as comparison maps in the variational inequalities satisfied by $u$ and $\tilde{u}$, respectively, and proceeding in a similar way as above.
However, most arguments can be simplified, since mollification with respect to time is not necessary in the present situation.
This allows us to deduce the following maximum principle.
\begin{lemma}[Maximum principle]
\label{lem:maximum_principle}
Let $T \in (0,\infty)$, assume that $\Omega\subset\mathbb{R}^{n}$ is open and bounded, and that $f \colon [0,T]\times\mathbb{R}^{n}\rightarrow\mathbb{R}$ satisfies \eqref{eq:integrand}. Consider $L\in(0,\infty]$ and functions $u,\tilde{u}\in K^{L}(\Omega_{T})$ such that $\partial_{t}u,\partial_{t}\tilde{u}\in L^{2}(\Omega_{T})$. Suppose moreover that $\left\Vert Du(0)\right\Vert _{L^{\infty}(\Omega,\mathbb{R}^{n})}$ and $\left\Vert D\tilde{u}(0)\right\Vert _{L^{\infty}(\Omega,\mathbb{R}^{n})}$ are bounded by $L$ if $L \in (0,\infty)$ and finite if $L=\infty$.
Finally, assume that for any $\tau \in (0,T]$ the function $u$ satisfies the variational inequality
\begin{align}\label{eq:max_principle 1}
\iint_{\Omega_{\tau}}f(t,Du)\,\dx\dt\leq & \iint_{\Omega_{\tau}} \partial_{t}v(v-u)+f(t,Dv)\,\dx\dt\nonumber\\
 & +\tfrac{1}{2}\left\Vert u(0)-v(0)\right\Vert^2 _{L^{2}(\Omega)}-\tfrac{1}{2}\left\Vert u(\tau)-v(\tau)\right\Vert^2 _{L^{2}(\Omega)}
\end{align}
whenever $v\in K^{L}(\Omega_{\tau})$ with $\partial_{t}v\in L^{2}(\Omega_{\tau})$
and $v=u$ on $\Omega \times (0,\tau)$, and that $\tilde{u}$ fulfills the analogical inequality.
Then
\[
\sup_{\Omega_{T}}(u-\tilde{u}) = \sup_{\partial_{\mathcal{P}}\Omega_{T}}(u-\tilde{u}).
\]
\end{lemma}
\begin{proof}
Let $\tau \in (0,T]$. Define
\[
	\hat{u}:=\tilde{u}+\sup_{\partial_{\mathcal{P}}\Omega_{T}}(u-\tilde{u}).
\]
Then $\hat{u}$ satisfies the variational inequality \eqref{eq:max_principle 1} with its own boundary values, and
\begin{equation}
	u\leq \hat{u}\quad\text{on}\quad\partial_{\mathcal{P}}\Omega_{T}. \label{eq:max_principle 2}
\end{equation}
Consider the functions $v:=\min(u,\hat{u})$ and $w:=\max(u,\hat{u})$.
Then $v,w\in K^{L}(\Omega_{\tau})$ with  $\partial_t v, \partial_t w \in L^2(\Omega_\tau)$ and $v=u$, $w=\hat{u}$ on $\partial\Omega\times(0,\tau)$.
Observe also that $v-u=-(u-\hat{u})_{+}$ and $w-\hat{u}=(u-\hat{u})_{+}$.
Using $v$ and $w$ as comparison functions for $u$ and $\hat{u}$
in the variational inequality \eqref{eq:max_principle 1}, we obtain
\begin{align*}
0& \leq \iint_{\Omega_{\tau}}\partial_{t}v(v-u)+\partial_{t}w(w-\hat{u})\,\dx\dt\\
 	&\phantom{=}+\iint_{\Omega_{\tau}}f(t,Dv)-f(t,Du)+f(t,Dw)-f(t,D\hat{u})\,\dx\dt\\
	&\phantom{=}+\tfrac{1}{2}\left\Vert (v-u)(0)\right\Vert _{L^{2}(\Omega)}^{2}+\tfrac{1}{2}\left\Vert (w-\hat{u})(0)\right\Vert _{L^{2}(\Omega)}^{2}\\
	&\phantom{=}-\tfrac{1}{2}\left\Vert (v-u)(\tau)\right\Vert _{L^{2}(\Omega)}^{2}-\tfrac{1}{2}\left\Vert (w-\hat{u})(\tau)\right\Vert _{L^{2}(\Omega)}^{2}\\
	&=\iint_{\Omega_{\tau}}\tfrac{1}{2}\partial_{t}((u-\hat{u})_{+})^{2}\,\dx\dt-\left\Vert (u-\hat{u})_{+}(\tau)\right\Vert _{L^{2}(\Omega)}^{2}\\
 & =-\tfrac{1}{2}\left\Vert (u-\hat{u})_{+}(\tau)\right\Vert _{L^{2}(\Omega)}^{2},
\end{align*}
where we used that $(v-u)(0)=(w-\hat{u})(0)=0$ and that the
terms with $f$ cancel one another. As $\tau$ was arbitrary, we
obtain 
\[
	u\leq\hat{u}=\tilde{u}+\sup_{\partial_{\mathcal{P}}\Omega_{T}}(u-\hat{u})\quad\text{in }\Omega_T
\]
so that
\[
	\sup_{\Omega_{T}}(u-\tilde{u})\leq\sup_{\partial_{\mathcal{P}}\Omega_{T}}(u-\tilde{u}).
\]
Since the reverse inequality holds by continuity, this proves the claim.
\end{proof}

\begin{lemma}[Localization in space]
\label{lem:spatial_localization}
Let $T \in (0,\infty)$, assume that $\Omega\subset\mathbb{R}^{n}$ is open and bounded, and that $f \colon [0,T]\times\mathbb{R}^{n}\rightarrow\mathbb{R}$ satisfies
\eqref{eq:integrand}. Consider $u_o\in W^{1,\infty}(\Omega)$
and $L\in(0,\infty]$ such that $\left\Vert Du_o\right\Vert _{L^{\infty}(\Omega,\mathbb{R}^{n})}\leq L$.
Suppose that $u$ is a variational solution to \eqref{eq:pde} in $K_{u_o}^{L}(\Omega_{T})$,
$L\in(0,\infty]$ (in the sense of Definition \ref{def:constrained_solution}
if $L<\infty$, in the sense of Definition \ref{def:unconstrainded_solution}
if $L=\infty$). Moreover, suppose that $\partial_{t}u\in L^{2}(\Omega_{T})$.
Then for any domain $\Omega^{\prime}\subset\Omega$ and any $\tau \in (0,T]$, the variational inequality
\begin{align}
	\iint_{\Omega_{\tau}^{\prime}}f(t,Du)\,\dx\dt &\leq \iint_{\Omega_{\tau}^{\prime}}\partial_{t}v(v-u)+f(t,Dv)\,\dx\dt\nonumber \label{eq:spatial_localization 1}\\
	 &\phantom{=}+\tfrac{1}{2}\left\Vert u(0)-v(0)\right\Vert^2 _{L^{2}(\Omega^{\prime})}-\tfrac{1}{2}\left\Vert u(\tau)-v(\tau)\right\Vert^2 _{L^{2}(\Omega^{\prime})}
\end{align}
holds whenever $v\in K^{L}_{u_o}(\Omega_{\tau}^{\prime})$ with $\partial_{t}v\in L^{2}(\Omega_{\tau})$
and $v=u$ on $\partial \Omega ^\prime \times (0,\tau)$.
\end{lemma}
\begin{proof}
By Lemma \ref{lem:local time} the function $u|_{\Omega_\tau}$ is a variational solution to \eqref{eq:pde} in the function space $K_{u_o}^L(\Omega_\tau)$. Observe that 
\[
	w:=\begin{cases}
		v & \text{in }\Omega_{\tau}^{\prime},\\
		u & \text{in }(\Omega\setminus\Omega^{\prime})_{\tau},
\end{cases}
\]
is an admissible comparison function for  $u|_{\Omega_\tau}$  in the variational inequality.
Inserting $w$ into the variational inequality \eqref{eq:var_ineq_constrained} if $L<\infty$ (or \eqref{eq:var_ineq_unconstrained} if $L=\infty$) with $T$ replaced by $\tau$ immediately yields \eqref{eq:spatial_localization 1}.
\end{proof}

\section{Existence for the gradient constrained problem for regular integrands}\label{sec:existence_for_constrained}
In this section, we are concerned with integrands that admit a time derivative.
More precisely, we consider $f \colon [0,T] \times \R^n \to \R$ such that
\begin{equation}
	\left\{
	\begin{array}{l}
		\mbox{$\xi \mapsto f(t,\xi)$ is convex for any $t \in [0,T]$,} \\[5pt]
		\mbox{$t \mapsto f(t,\xi) \in W^{1,1}(0,T)$ for any $\xi \in \R^n$,} \\[5pt]
		\mbox{for any $L>0$  there exists $\tilde{g}_L \in L^1(0,T)$ such that $|\partial_t f(t,\xi)| \leq \tilde{g}_L(t)$} \\
		\mbox{for a.e.~$t \in [0,T]$ and all $\xi \in B_L(0)$.}
	\end{array}
	\right.
	\label{eq:regular_f}
\end{equation}
The aim of this section is to prove the following existence result.
\begin{theorem}
\label{thm:existence_constrained}
Let $\Omega \subset \R^n$ be a bounded Lipschitz domain and $T \in (0,\infty)$.
Consider a boundary datum $u_o \in W^{1,\infty}(\Omega)$ such that $\|Du_o\|_{L^\infty(\Omega,\R^n)} \leq L$ for a constant $L \in (0, \infty)$.
Further, assume that the integrand $f \colon [0,T] \times \R^n \to \R$ satisfies hypothesis \eqref{eq:regular_f}.
Then, there exists a variational solution $u \in K^L_{u_o}(\Omega_T)$ to the gradient constrained problem in the sense of Definition~\ref{def:constrained_solution}.
Further, there holds $\partial_t u \in L^2(\Omega_T)$ with the quantitative bound
$$
	\iint_{\Omega_T} |\partial_t u|^2 \,\dx\dt
	\leq
	4|\Omega| \big( \sup_{|\xi| \leq L} |f(0,\xi)| + \|\tilde{g}_L\|_{L^1(0,T)} \big).
$$
\end{theorem}

We prove Theorem~\ref{thm:existence_constrained} via the method of minimizing movements.
The proof is divided into five steps.

\subsection{A sequence of minimizers to elliptic variational functionals}
Fix a step size $h := \frac{T}{m}$ for some $m \in \N$ and consider times
$ih$, $i = 0,\ldots,m$.
For $i=0$, set $u_0 := u_o \in W^{1,\infty}(\Omega)$ with
$\|Du_o\|_{L^\infty(\Omega,\R^n)} \leq L$.
Further, for $i = 1,\ldots,m$, $u_i$ is defined as the minimizer of the elliptic variational functional
$$
	F_i[v] :=
	\int_\Omega f(ih, Dv) \,\dx + \tfrac{1}{2h} \int_\Omega |v - u_{i-1}|^2 \,\dx
$$
in the class
$\mathcal{A} := \{v \in W^{1,\infty}(\Omega) : \text{$v=u_o$ on $\partial\Omega$ and $\|Dv\|_{L^\infty(\Omega,\R^n)} \leq L$} \}$.
The existence of a minimizer to $F_i$ in this class is ensured by the direct method in the calculus of variations.
More precisely, note that $\mathcal{A} \neq \emptyset$, since $u_o \in \mathcal{A}$, and consider a minimizing sequence to $F_i$ in $\mathcal{A}$, i.e. a sequence $(u_{i,j})_{j \in \N} \subset \mathcal{A}$ such that
$$
	\lim_{j \to \infty} F_i[u_{i,j}]
	=
	\inf_{v \in \mathcal{A}} F_i[v].
$$
%Observe that $\mathcal{A}$ is equicontinuous.
%Choosing an arbitrary $x_o \in \partial\Omega$ and estimating
%\begin{align*}
%	|v(x)|
%	\leq
%	|v(x_o)| + |v(x) - v(x_o)|
%	\leq
%	|u(x_o)| + L|x-x_o|
%	\leq
%	\sup_{\partial\Omega} |u_o| + L\diam(\Omega)
%\end{align*}
%for any $v \in \mathcal{A}$ shows that $\mathcal{A}$ is equibounded.
Further, by definition of $\mathcal{A}$ and Rellich's theorem there exists a limit map $u_i \in \mathcal{A}$ and a (not relabelled) subsequence such that
$$
	\left\{
	\begin{array}{l}
		\mbox{$u_{i,j} \to u_i$ strongly in $L^2(\Omega)$ as $j \to \infty$,} \\[5pt]
		\mbox{$Du_{i,j} \wto Du_i$ weakly in $L^2(\Omega,\R^n)$ as $j \to \infty$.}
	\end{array}
	\right.
$$
Since the functional $\widetilde{F}_i \colon W^{1,2}(\Omega) \to (-\infty,\infty]$,
$$
	\widetilde{F}_i[v] :=
	\left\{
	\begin{array}{ll}
		F_i[v] &\text{if } v \in \mathcal{A}, \\[5pt]
		\infty &\text{else}
	\end{array}
	\right.
$$
is proper, convex and lower semicontinuous with respect to strong convergence in $W^{1,2}(\Omega)$, it is also lower semicontinuous with respect to weak convergence in $W^{1,2}(\Omega)$, see~\cite[Corollary 2.2]{Ekeland-Temam}.
Therefore, we obtain that
$$
	F_i[u_i]
	=
	\widetilde{F}_i[u_i]
	\leq
	\liminf_{j \to \infty} \widetilde{F}_i[u_{i,j}]
	=
	\lim_{j \to \infty} F_i[u_{i,j}]
	=
	\inf_{v \in \mathcal{A}} F_i[v].
$$

\subsection{Energy estimates}
Since $u_{i-1} \in \mathcal{A}$ is an admissible comparison map for the minimizer $u_i$
and $f$ fulfills \eqref{eq:regular_f}$_3$, we have that
\begin{align*}
	\int_\Omega &f(ih,Du_i) \,\dx + \tfrac{1}{2h} \int_\Omega |u_i - u_{i-1}|^2 \,\dx 
	=
	F_i[u_i] \\
	&\leq
	F_i[u_{i-1}] \\
	&=
	\int_\Omega f((i-1)h,Du_{i-1}) \,\dx
	+ \int_\Omega f(ih,Du_{i-1}) - f((i-1)h,Du_{i-1}) \,\dx \\
	&\leq
	\int_\Omega f((i-1)h,Du_{i-1}) \,\dx
	+ \iint_{\Omega \times ((i-1)h,ih)} |\partial_t f(t,Du_{i-1})| \,\dx\dt \\
	&\leq
	\int_\Omega f((i-1)h,Du_{i-1}) \,\dx
	+ |\Omega| \int_{((i-1)h,ih)} |\tilde{g}_L(t)| \,\dt.
\end{align*}
Summing up the preceding inequalities from $i=1$ to $i=m$, we find that
\begin{align*}
	\sum_{i=1}^m \int_\Omega &f(ih,Du_i) \,\dx\dt
	+ \tfrac{1}{2h} \sum_{i=1}^m \int_\Omega |u_i - u_{i-1}|^2 \,\dx \\
	&\leq
	\sum_{i=1}^m \int_\Omega f((i-1)h,Du_{i-1}) \,\dx
	+ |\Omega| \int_{(0,T)} |\tilde{g}_L(t)| \,\dt.
\end{align*}
Subtracting the first term on the left-hand side, we conclude that
\begin{align}
	\tfrac{1}{2h} \sum_{i=1}^m \int_\Omega |u_i - u_{i-1}|^2 \,\dx
	&\leq
	\int_\Omega f(0,Du_o) \,\dx - \int_\Omega f(T,Du_m) \,\dx
	+ |\Omega| \|\tilde{g}_L\|_{L^1(0,T)} \nonumber \\
	&\leq
	2 |\Omega| \big( \sup_{|\xi| \leq L} |f(0,\xi)| + \|\tilde{g}_L\|_{L^1(0,T)} \big).
	\label{eq:energy_estimate}
\end{align}

\subsection{The limit map}
In the following we denote the step size by $h_m$ in order to emphasize the dependence on $m$.
First, we join the minimizers $u_i$ to a map that is piecewise constant with respect to time.
More precisely, we define $u^{(m)} \colon \Omega \times (-h_m,T] \to \R$ by
$$
	u^{(m)}(t) := u_i
	\quad \text{for } t \in ((i-1)h_m, ih_m], \, i=0,\ldots,m.
$$
Observe that the sequence $\big(u^{(m)} \big)_{m \in \N}$ is bounded in $L^\infty(\Omega_T)$,
since $\|u^{(m)}\|_{L^\infty(\Omega_T)} = \max_{i=0,\ldots,m} \|u_i\|_{L^\infty(\Omega)}$,
$u_i \in \mathcal{A}$ for all $i=0,\ldots,m$ and $\mathcal{A}$ is equibounded.
Further, we know that $\|Du^{(m)}\|_{L^\infty(\Omega_T,\R^n)} = \max_{i=0,\ldots,m} \|
Du_i\|_{L^\infty(\Omega,\R^n)} \leq L$ for any $m \in \N$.
Therefore, there exists a subsequence $\mathfrak{K} \subset \N$ and a limit map $u \in L^\infty(\Omega_T)$ such that $\|Du\|_{L^\infty(\Omega_T,\R^n)} \leq L$, $u= u_o$ on $\partial\Omega \times (0,T)$ and
\begin{equation}
	\left\{
	\begin{array}{l}
	\mbox{$u^{(m)} \wsto u$ weakly$^*$ in $L^\infty(\Omega_T)$ as $\mathfrak{K} \ni m \to \infty$,} \\[5pt]
	\mbox{$u^{(m)}(t) \to u(t)$ uniformly as $\mathfrak{K} \ni m \to \infty$ for each $t \in [0,T]$,} \\[5pt]
	\mbox{$Du^{(m)} \wsto Du$ weakly$^*$ in $L^\infty(\Omega_T,\R^n)$ as $\mathfrak{K} \ni m \to \infty$.}
	\end{array}
	\right.
	\label{eq:convergence_um}
\end{equation}
In order to prove that $u$ has a time derivative, we consider the linear interpolation of minimizers $\tilde{u}^{(m)} \colon \Omega \times (-h_m,T] \to \R$
given by $\tilde{u}^{(m)}(t) := u_o$ for $t \in (-h_m,0]$ and
$$
	\tilde{u}^{(m)}(t) :=
	\Big(i - \tfrac{t}{h_m} \Big) u_{i-1} + \Big( 1 - i + \tfrac{t}{h_m} \Big) u_i
	\quad \text{for } t \in ((i-1)h_m, ih_m], \, i=1,\ldots,m.
$$
Similar arguments as above ensure that $\big(\tilde{u}^{(m)} \big)_{m \in \N}$ is bounded in $L^\infty(\Omega_T)$ and that $\|D\tilde{u}^{(m)}\|_{L^\infty(\Omega_T,\R^n)} \leq L$ for any $m \in \N$.
Moreover, by the energy bound \eqref{eq:energy_estimate} we obtain that
\begin{align}
	\iint_{\Omega_T} |\partial_t \tilde{u}^{(m)}|^2 \,\dx\dt
	&=
	\sum_{i=1}^m \iint_{\Omega \times ((i-1)h_m, ih_m]}
	\tfrac{1}{h_m^2} |u_i - u_{i-1}|^2 \,\dx\dt \nonumber \\
	&=
	\tfrac{1}{h_m} \sum_{i=1}^m \int_\Omega |u_i - u_{i-1}|^2 \,\dx \nonumber \\
	&\leq
	4|\Omega| \big( \sup_{|\xi| \leq L} |f(0,\xi)| + \|\tilde{g}_L\|_{L^1(0,T)} \big).
	\label{eq:time_derivative_um}
\end{align}
Hence, $\big( \tilde{u}^{(m)} \big)_{m \in \N}$ is bounded in $W^{1,2}(\Omega_T)$.
By Rellich's theorem we conclude that there exists a subsequence still labelled $\mathfrak{K}$ and a limit map $\tilde{u} \in L^\infty(\Omega_T)$ with $\|D\tilde{u}\|_{L^\infty(\Omega_T,\R^n)} \leq L$, $\tilde{u}= u_o$ on $\partial\Omega \times (0,T)$
and $\partial_t \tilde{u} \in L^2(\Omega_T)$ such that
\begin{equation}
	\left\{
	\begin{array}{l}
	\mbox{$\tilde{u}^{(m)} \to u$ strongly in $L^2(\Omega_T)$ as $\mathfrak{K} \ni m \to \infty$,} \\[5pt]
%	\mbox{$D\tilde{u}^{(m)} \wsto u$ weakly$^*$ in $L^\infty(\Omega_T,\R^n)$ as $\mathfrak{K} \ni m \to \infty$,} \\[5pt]
	\mbox{$\partial_t \tilde{u}^{(m)} \wto \partial_t \tilde{u}$ weakly in $L^2(\Omega_T)$ as $\mathfrak{K} \ni m \to \infty$.}
	\end{array}
	\right.
	\label{eq:convergence_tilde_um}
\end{equation}
Note that $\partial_t \tilde{u} \in L^2(\Omega_T)$ in particular implies that $\tilde{u} \in C^{0;\frac{1}{2}}([0,T];L^2(\Omega))$ and therefore $\tilde{u}$ is contained in the class of functions $K_{u_o}^L(\Omega_T)$.
Next, since $\big| \big( u^{(m)} - \tilde{u}^{(m)} \big)(t) \big| \leq |u_i - u_{i-1}|$ for $t  \in ((i-1)h_m, ih_m]$, $i=1,\ldots,m$, we infer from \eqref{eq:energy_estimate} that
\begin{align*}
	\iint_{\Omega_T} \big| u^{(m)} - \tilde{u}^{(m)} \big|^2 \,\dx\dt
	&\leq
	h_m \sum_{i=1}^m \int_\Omega |u_i - u_{i-1}|^2 \,\dx \\
	&\leq
	4|\Omega| \big( \sup_{|\xi| \leq L} |f(0,\xi)| + \|\tilde{g}_L\|_{L^1(0,T)} \big) h_m^2.
\end{align*}
Together with \eqref{eq:convergence_tilde_um}$_1$ this implies that $u^{(m)} \to \tilde{u}$ strongly in $L^2(\Omega_T)$ as $\mathfrak{K} \ni m \to \infty$ and thus in particular that $u = \tilde{u} \in K_{u_o}^L(\Omega_T)$ with $\partial_t u \in L^2(\Omega_T)$.
Finally, by lower semicontinuity with respect to weak convergence, \eqref{eq:time_derivative_um} gives us the claimed bound
$$
	\iint_{\Omega_T} |\partial_t u|^2 \,\dx\dt
	\leq
	4|\Omega| \big( \sup_{|\xi| \leq L} |f(0,\xi)| + \|\tilde{g}_L\|_{L^1(0,T)} \big).
$$

\subsection{Minimizing property of the approximations}
First, define piecewise constant approximations of the integrand by
$$
	f^{(m)}(t,\xi) := f(ih,\xi)
	\quad \text{for } t  \in ((i-1)h_m, ih_m], \, i=0, \ldots, m.
$$
We claim that $u^{(m)}$ is a minimizer of the functional
$$
	F^{(m)}[v] :=
	\iint_{\Omega_T} f^{(m)}(t,Dv) \,\dx\dt
	+ \tfrac{1}{2h_m} \iint_{\Omega_T} |v(t) - u^{(m)}(t-h_m)|^2 \,\dx\dt
$$
in the class of functions
$$
	\mathcal{A}_T :=
	\{v \in L^\infty(\Omega_T) :
	\| Du \|_{L^\infty(\Omega_T,\R^n)} \leq L
	\text{ and $u=u_o$ on } \partial\Omega \times (0,T) \}.
$$
Indeed, consider an arbitrary map $v \in \mathcal{A}_T$.
Since $v(t) \in \mathcal{A}$ for a.e.~$t \in [0,T]$, by the minimizing property of $u_i$ with respect to $F_i$ in the class $\mathcal{A}$ we find that
\begin{align*}
	F^{(m)} \big[ u^{(m)} \big]
	&=
	\sum_{i=1}^m \int_{((i-1)h_m, ih_m ]} F_i[u_i] \,\dt 
	\leq
	\sum_{i=1}^m \int_{((i-1)h_m, ih_m ]} F_i[v(t)] \,\dt
	=
	F^{(m)}[v].
\end{align*}
A straightforward computation shows that this is equivalent to
\begin{align*}
	\iint_{\Omega_T} &f^{(m)} \big( t, Du^{(m)} \big) \,\dx\dt \\
	&\leq
	\iint_{\Omega_T} f^{(m)}(t,Dv) \,\dx\dt \\
	&\phantom{=}
	+ \tfrac{1}{h_m} \iint_{\Omega_T} \tfrac12 \big| v - u^{(m)} \big|^2
	+ \big( v - u^{(m)} \big) \big( u^{(m)} - u^{(m)}(t-h_m) \big) \,\dx\dt
\end{align*}
for any $v \in \mathcal{A}_T$.
Choosing the convex combination $u^{(m)} + s \big( v - u^{(m)} \big) \in \mathcal{A}_T$
with $s \in (0,1)$ as comparison map and using the convexity of $\xi \mapsto f(t,\xi)$
for all $t \in [0,T]$, we obtain that
\begin{align*}
	\iint_{\Omega_T} &f^{(m)} \big( t, Du^{(m)} \big) \,\dx\dt \\
	&\leq
	(1-s) \iint_{\Omega_T} f^{(m)} \big( t, Du^{(m)} \big) \,\dx\dt
	+s \iint_{\Omega_T} f^{(m)}(t, Dv) \,\dx\dt \\
	&\phantom{=}
	+ \tfrac{1}{h_m} \iint_{\Omega_T} \tfrac{s^2}{2} \big| v - u^{(m)} \big|^2
	+ s\big( v - u^{(m)} \big) \big( u^{(m)} - u^{(m)}(t-h_m) \big) \,\dx\dt.
\end{align*}
Reabsorbing the first term on the right-hand side into the left-hand side, dividing the resulting inequality by $s$ and taking the limit $s \downarrow 0$ gives us that
\begin{align*}
	\iint_{\Omega_T} &f^{(m)} \big( t, Du^{(m)} \big) \,\dx\dt \\
	&\leq
	\iint_{\Omega_T} f^{(m)}(t, Dv) \,\dx\dt
	+ \tfrac{1}{h_m} \iint_{\Omega_T} \big( v - u^{(m)} \big) \big( u^{(m)} - u^{(m)}(t-h_m) \big) \,\dx\dt.
\end{align*}
Next, assume without loss of generality that $v(0) \in L^\infty(\Omega)$,
extend $v$ to $(-h_m, 0]$ by $v(0)$ and note that
\begin{align*}
	\big( v &- u^{(m)} \big) \big( u^{(m)} - u^{(m)}(t-h_m) \big) \\
	&=
	\big( v - u^{(m)} \big) \big( v - v(t-h_m) \big)
	+ \tfrac12 \big( v(t-h_m) - u^{(m)}(t-h_m) \big)^2
	-\tfrac12 \big( v - u^{(m)} \big)^2 \\
	&\phantom{=}
	- \tfrac12 \big( v - v(t-h_m) - u^{(m)} + u^{(m)}(t-h_m) \big)^2 \\
	&\leq
	\big( v - u^{(m)} \big) \big( v - v(t-h_m) \big)
	+ \tfrac12 \big( v(t-h_m) - u^{(m)}(t-h_m) \big)^2
	-\tfrac12 \big( v - u^{(m)} \big)^2.
\end{align*}
Inserting this into the preceding inequality and recalling that $v(t) = v(0)$ for $t \in (-h_m, 0]$, we infer
\begin{align}
	\iint_{\Omega_T} &f^{(m)} \big( t, Du^{(m)} \big) \,\dx\dt \nonumber\\
	&\leq
	\iint_{\Omega_T} f^{(m)}(t, Dv) \,\dx\dt
	+ \tfrac{1}{h_m} \iint_{\Omega_T} \big( v - u^{(m)} \big) \big( v - v(t-h_m) \big) \,\dx\dt
	\label{eq:preliminary_var_ineq} \\
	&\phantom{=}
	+\tfrac{1}{2h_m} \iint_{\Omega_T} \big( v(t-h_m) - u^{(m)}(t-h_m) \big)^2
	-\big( v - u^{(m)} \big)^2 \,\dx\dt \nonumber \\
	&=
	\iint_{\Omega_T} f^{(m)}(t, Dv) \,\dx\dt
	+ \tfrac{1}{h_m} \iint_{\Omega_T} \big( v - u^{(m)} \big) \big( v - v(t-h_m) \big) \,\dx\dt \nonumber \\
	&\phantom{=}
	+\tfrac{1}{2} \int_{\Omega} (v-u_o)^2 \,\dx
	-\tfrac{1}{2h_m} \iint_{\Omega \times (T-h_m, T]} \big| v - u^{(m)}(T) \big|^2 \,\dx\dt. \nonumber
\end{align}

\subsection{Variational inequality for the limit map}
We fix an arbitrary map $v \in K^L_{u_o}(\Omega_T)$ with $\partial_t v \in L^2(\Omega_T)$.
Thus, in particular we have that $v \in \mathcal{A}_T$, so $v$ is an admissible comparison map in \eqref{eq:preliminary_var_ineq}.
Our goal is to pass to the limit $\mathfrak{K} \ni m \to \infty$ in \eqref{eq:preliminary_var_ineq} in order to deduce the variational inequality \eqref{eq:var_ineq_constrained} for $u$.
To this end, we consider the terms separately.
First, we write the first term on the left-hand side of \eqref{eq:preliminary_var_ineq} as
\begin{align*}
	\iint_{\Omega_T} &f^{(m)} \big( t, Du^{(m)} \big) \,\dx\dt \\
	&=
	\iint_{\Omega_T} f\big( t, Du^{(m)} \big) \,\dx\dt
	+ \iint_{\Omega_T} f^{(m)} \big( t, Du^{(m)} \big) - f\big( t, Du^{(m)} \big) \,\dx\dt.
\end{align*}
By Lemma \ref{lem:lower_semicontinuity} and \eqref{eq:convergence_um}$_3$, we obtain that
$$
	\iint_{\Omega_T} f(t,Du) \,\dx\dt
	\leq
	\liminf_{\mathfrak{K} \ni m \to \infty}
	\iint_{\Omega_T} f\big( t, Du^{(m)} \big) \,\dx\dt.
$$
Further, since $\big\| Du^{(m)} \big\|_{L^\infty(\Omega_T,\R^n)} \leq L$ for all $m \in \N$ and $f$ fulfills \eqref{eq:regular_f}$_3$, we estimate
\begin{align*}
	\bigg| \iint_{\Omega_T} &f^{(m)} \big( t, Du^{(m)} \big) - f\big( t, Du^{(m)} \big) \,\dt \bigg| \\
	&\leq
	\sum_{i=1}^m \iint_{\Omega \times ((i-1)h_m, ih_m]} \big| f\big( ih_m, Du^{(m)} \big) - f\big( t, Du^{(m)} \big) \big| \dx\dt \\
	&\leq
	\sum_{i=1}^m \iint_{\Omega \times ((i-1)h_m, ih_m]} \int_{((i-1)h_m, ih_m]} \big|\partial_t f\big( s, Du^{(m)}(t) \big) \big| \, \ds \, \dx\dt \\
	&\leq
	|\Omega| h_m \sum_{i=1}^m \int_{((i-1)h_m, ih_m]} \tilde{g}_L(s) \,\ds \\
	&=
	|\Omega| \|\tilde{g}_L \|_{L^1(0,T)} h_m.
\end{align*}
Therefore, this term vanishes in the limit $m \to \infty$.
Joining the preceding estimates, we conclude that
\begin{equation}
	\iint_{\Omega_T} f(t,Du) \,\dx\dt
	\leq
	\liminf_{\mathfrak{K} \ni m \to \infty}
	\iint_{\Omega_T} f^{(m)}\big( t, Du^{(m)} \big) \,\dx\dt.
	\label{eq:mm_aux1}
\end{equation}
Repeating the estimates in the penultimate inequality with $u^{(m)}$ replaced by $v$, for the first term on the right-hand side of \eqref{eq:preliminary_var_ineq} we find that
\begin{equation}
	\iint_{\Omega_T} f(t,Dv) \,\dx\dt
	=
	\lim_{m \to \infty} \iint_{\Omega_T} f^{(m)}(t,Dv) \,\dx\dt.
	\label{eq:mm_aux2}
\end{equation}
Next, since $\frac{1}{h_m} (v(t) - v(t-h_m)) \to \partial_t v$ strongly in $L^2(\Omega_T)$ and $u^{(m)} \wto u$ weakly in $L^2(\Omega_T)$ as $\mathfrak{K} \ni m \to \infty$ by \eqref{eq:convergence_um}$_1$, we have that
\begin{equation}
	\iint_{\Omega_T} \partial_t v (v-u) \,\dx\dt
	=
	\lim_{\mathfrak{K} \ni m \to \infty}
	\tfrac{1}{h_m} \iint_{\Omega_T} \big( v - u^{(m)} \big) \big( v - v(t-h_m) \big) \,\dx\dt.
	\label{eq:mm_aux3}
\end{equation}
Finally, by the fact that $v \in C^0([0,T];L^2(\Omega))$ and by \eqref{eq:convergence_um}$_2$, we obtain that
\begin{equation}
	\tfrac12 \|(v-u)(T)\|_{L^2(\Omega)}^2
	=
	\lim_{\mathfrak{K} \ni m \to \infty}
	\tfrac{1}{2h_m} \iint_{\Omega \times (T-h_m, T]} \big| v - u^{(m)}(T) \big|^2 \,\dx\dt.
	\label{eq:mm_aux4}
\end{equation}
Collecting the assertions \eqref{eq:mm_aux1} -- \eqref{eq:mm_aux4} yields
\begin{align*}
	\iint_{\Omega_T} f(t,Du) \,\dx\dt
	&\leq
	\iint_{\Omega_T} f(t,Dv) \,\dx\dt
	+\iint_{\Omega_T} \partial_t v (v-u) \,\dx\dt \\
	&\phantom{=}
	+\tfrac12 \| v(0) - u_o \|_{L^2(\Omega)}^2
	-\tfrac12 \|(v-u)(T)\|_{L^2(\Omega)}^2.
\end{align*}
Since $v \in K^L_{u_o}(\Omega_T)$ with $\partial_t v \in L^2(\Omega_T)$ was arbitrary, we have shown that $u \in K^L_{u_o}(\Omega_T)$ is the desired variational solution.
\hfill $\qed$

\section{Existence for the unconstrained problem for regular integrands}
\label{sec:unconstrained_bsc}
In this section we show the existence of variational solutions to the unconstrained problem under the regularity condition \eqref{eq:regular_f} provided that the initial and boundary datum satisfies the bounded slope condition.
To this end, we need the following lemma, whose proof is similar to that of \cite[Lemma 7.1]{BDMS}. It states that affine functions independent of time are variational solutions to \eqref{eq:pde} with respect to their own initial and lateral boundary values.

\begin{lemma}
\label{lem:affine_is_sol}
Let $\Omega$ be open and bounded. Assume that $f \colon [0,T]\times\mathbb{R}^n \rightarrow\mathbb{R}$ satisfies \eqref{eq:integrand}. Let $w(x,t):=a+\xi\cdot x$ with constants $a\in\mathbb{R}$ and $\xi\in\mathbb{R}^{n}$ be an affine function independent of time. Then $w$ is a variational solution in the sense of Definition~\ref{def:unconstrainded_solution} in $K_{w}^{\infty}(\Omega_{T})$.
\end{lemma}

With the preceding lemma at hand, we are able to prove the following.
\begin{theorem}
\label{thm:existence_unconstrained_regular}
Let $T \in (0,\infty)$, assume that $\Omega\subset\mathbb{R}^{n}$ is open, bounded and
convex, and that the integrand $f \colon [0,T]\times\mathbb{R}^{n}\rightarrow\mathbb{R}$
satisfies \eqref{eq:regular_f}. Consider $u_o\in W^{1,\infty}(\Omega)$
such that $\left\Vert Du_o\right\Vert _{L^{\infty}(\Omega,\mathbb{R}^{n})}\leq Q$
and suppose that $\left. u_o \right|_{\partial\Omega}$ satisfies the bounded slope
condition with the same parameter $Q$. Then there exists a variational
solution $u\in K^{\infty}_{u_o}(\Omega_{T})$ to \eqref{eq:pde} in the sense
of Definition \ref{def:unconstrainded_solution}. Further, we have
the quantitative bound
\begin{equation}
\left\Vert Du\right\Vert _{L^{\infty}(\Omega_{T},\mathbb{R}^{n})}\leq Q. \label{eq:existence_unconstrained 1}
\end{equation}
\end{theorem}
\begin{proof}
Let $L > Q$. By Theorem \ref{thm:existence_constrained} there exists a variational solution $u \in K^L_{u_o} (\Omega_T)$ with $\partial_t u \in L^2(\Omega_T)$ to the gradient constrained problem in the sense of Definition~\ref{def:constrained_solution}.
We begin by proving the Lipschitz bound \eqref{eq:existence_unconstrained 1} and then show that $u$ is in fact already a solution to the unconstrained problem.

Fix $x_o \in\partial\Omega$ and denote by $w_{x_o}^{\pm}$ the
affine functions from Lemma \ref{lem:bounded_slope}. In particular we have $w_{x_o}^{-}\leq u_o\leq w_{x_o}^{+}$.
Since by Lemma \ref{lem:affine_is_sol} the functions $w_{x_o}^{-}$
and $w_{x_o}^{+}$ are variational solutions, it follows from the
comparison principle in Theorem \ref{thm:comparison principle} that 
\[
	w_{x_o}^{+}(x)\leq u(x,t)\leq w_{x_o}^{-}(x)\quad\text{for all }(x,t)\in\Omega_{T}.
\]
Consequently, there holds
\[
	\left|u(x,t)-u_o(x_o)\right|\leq Q\left|x-x_o\right|\quad\text{for all }(x,t)\in\Omega_{T}.
\]
Since $x_o\in\partial\Omega$ was arbitrary, we obtain that
\begin{equation}
\frac{\left|u(x,t)-u_o(x_o)\right|}{\left|x-x_o\right|}\leq Q\quad\text{for all }x_o \in\partial\Omega,(x,t)\in\Omega_{T}.\label{eq:constrained estimate 1}
\end{equation}
Consider $x_{1},x_{2}\in\Omega$, $x_{1}\not=x_{2}$, $t\in(0,T)$
and set $y:=x_{2}-x_{1}$. Define the shifted set $\widetilde{\Omega}_{T}:=\left\{ (x-y,t)\in\mathbb{R}^{n+1}:(x,t)\in\Omega_{T}\right\} $
and the shifted function $u_{y}\colon\widetilde{\Omega}_{T}\rightarrow\mathbb{R}$
by
\[
	u_{y}(x,t):=u(x+y,t).
\]
Then $u_{y}$ is a variational solution in $K^{L}(\widetilde{\Omega}_{T})$.
Since $\partial_{t}u,\partial_{t}u_{y}\in L^{2}((\Omega\cap\widetilde{\Omega})_{T})$ by the spatial localization principle in Lemma \ref{lem:spatial_localization}, the functions $u$ and $u_{y}$ both satisfy variational inequality \eqref{eq:max_principle 1} from Lemma \ref{lem:maximum_principle} in $(\Omega\cap\widetilde{\Omega})_T$. Therefore by Lemma \ref{lem:maximum_principle} there exists $(x_o,t_o)\in\mathcal{\partial_{\mathcal{P}}}((\Omega\cap\widetilde{\Omega}))_{T}$ such that 
\begin{align*}
	\left|u(x_{1},t)-u_{y}(x_{1},t)\right| & \leq\left|u(x_o,t_o)-u_{y}(x_o,t_o)\right|.
\end{align*}
By definition of $y$ and $u_{y}$, this yields
\begin{align*}
\left|u(x_{1},t)-u(x_{2},t)\right|\leq & \left|u(x_o,t_o)-u(x_o+y,t_o)\right|.
\end{align*}
Since either $t_o=0$ or one of the points $x_o$ or $x_o+y$ belongs to $\partial\Omega$, it follows from the assumption $\left\Vert Du_o\right\Vert _{L^{\infty}(\Omega,\mathbb{R}^{n})}\leq Q$
and (\ref{eq:constrained estimate 1}) that
\[
\left|u(x_o,t_o)-u(x_o+y,t_o)\right|\leq Q\left|y\right|=Q\left|x_1-x_2\right|.
\]
Combining this with the preceding estimate, we obtain \eqref{eq:existence_unconstrained 1}.

It remains to show that $u$ is a variational solution to the unconstrained problem. Let $w\in K^\infty_{u_o}(\Omega_{T})$ with $\partial_{t}w\in L^{2}(\Omega_{T})$
and choose the comparison map $v:=u+s(w-u)$ for $0<s\ll1$; in particular, since $Q < L$, for $s$ small enough we have that
\[
\left\Vert Dv\right\Vert _{L^{\infty}(\Omega_{T},\mathbb{R}^{n})}\leq\left\Vert Du\right\Vert _{L^{\infty}(\Omega_{T},\mathbb{R}^{n})}+s(\left\Vert Dw\right\Vert _{L^{\infty}(\Omega_{T},\mathbb{R}^{n})}+\left\Vert Du\right\Vert _{L^{\infty}(\Omega_{T},\mathbb{R}^{n})})\leq L.
\]
Thus $v$ is an admissible comparison function for the gradient constrained problem and we obtain that
\begin{align*}
\iint_{\Omega_{T}}f(t,Du)\,\dx\dt & \leq\iint_{\Omega_{T}}s\partial_{t}u(w-u)+sf(t,Dw)+(1-s)f(t,Du)\,\dx\dt \\
 & +\tfrac{s}{2}\left\Vert w(0)-u_o\right\Vert _{L^{2}(\Omega_{T})}^2
 -\tfrac{s}{2}\left\Vert w(T)-u(T)\right\Vert _{L^{2}(\Omega_{T})}^2
\end{align*}
Reabsorbing the integral with $f(t,Du)$ to the left-hand side and
dividing by $s$, we see that $u$ satisfies the variational inequality
\eqref{eq:var_ineq_unconstrained}. Thus $u$ is a variational solution in the sense of Definition \ref{def:unconstrainded_solution}.
\end{proof}

\section{Existence for the unconstrained problem for general integrands}\label{sec:unconstrained_for_general}

In this section we finish the proof of Theorem \ref{thm:main_existence}. Note that we only need to consider the case $T < \infty$.
Indeed, assume that for any $\tau \in (0,\infty)$ we have constructed a variational solution with initial and boundary datum $u_o$ in the sense of Definition \ref{def:unconstrainded_solution} such that the gradient bound \eqref{eq:gradient_bound} holds in $\Omega_\tau$.
Let $0 < \tau_1 < \tau_2 < \infty$ and denote by $u_1$ and $u_2$ the variational solutions in $\Omega_{\tau_1}$ and $\Omega_{\tau_2}$, respectively.
By the localization principle with respect to time in Lemma \ref{lem:local time},
$u_2$ is also a variational solution in $\Omega_{\tau_1}$.
Further, $u_1$ and $u_2$ coincide in $\Omega_{\tau_1}$ by the comparison principle in Theorem \ref{thm:comparison principle}.
Therefore, a unique global variational solution in the sense of Definition \ref{def:constrained_solution} can be constructed by taking an increasing sequence of times $(\tau_i)_{i \in \N}$ with $\lim_{i \to \infty} \tau_i = \infty$.

Thus we suppose that  $T < \infty$. For $\epsilon>0$ we define the Steklov average $f_\epsilon \colon [0,T] \times \R^n \to \R$ of $f$ by \eqref{eq:Steklov}.
A straightforward computation shows that $\xi \mapsto f_\epsilon(t,\xi)$ is convex for any $t \in [0,T]$.
Further, for any $\epsilon>0$ the derivative of $f_\epsilon$ with respect to the time variable is given by
$$
	\partial_t f(t,\xi)
	=
	\tfrac{1}{\epsilon} ( f(t+\epsilon,\xi) - f(t,\xi) ).
$$
Combining this with \eqref{eq:dominating_f}, for any $L>0$ we have that
$$
	|\partial_t f(t,\xi)|
	\leq 
	\tfrac{1}{\epsilon} (g_L(t+\epsilon) + g_L(t))
	\quad \text{for all } t \in [0,T], \xi \in B_L(0).
$$
Hence, for any $\epsilon>0$, the integrand $f_\epsilon$ fulfills assumption \eqref{eq:regular_f}.
By Theorem \ref{thm:existence_unconstrained_regular} we conclude that for any $\epsilon>0$ there exists a variational solution $u_\epsilon \in K^\infty_{u_o}(\Omega_T)$ to the Cauchy-Dirichlet problem associated with $f_\epsilon$ in the sense of Definition \ref{def:unconstrainded_solution} satisfying the bound
$$
	\| Du_\epsilon \|_{L^\infty(\Omega_T,\R^n)}
	\leq
	\max\{ Q, \| Du_o \|_{L^\infty(\Omega,\R^n)} \}.
$$
Together with the fact that $u_\epsilon = u_o$ on $\partial\Omega \times (0,T)$, this implies in particular that the sequence $(u_\epsilon)_{\epsilon>0}$ is bounded in $L^\infty(\Omega_T)$.
Thus, there exists a (not relabelled) subsequence and a limit map $u \in L^\infty(\Omega_T)$ such that $u = u_o$ on $\partial\Omega \times (0,T)$,
$$
	\| Du \|_{L^\infty(\Omega_T,\R^n)}
	\leq
	\max\{ Q, \| Du_o \|_{L^\infty(\Omega,\R^n)} \}
$$
and in the limit $\epsilon \downarrow 0$ there holds
\begin{equation}
	\left\{
	\begin{array}{l}
		\mbox{$u_\epsilon \wsto u$ weakly$^\ast$ in $L^\infty(\Omega_T)$,} \\[5pt]
		\mbox{$u_\epsilon(t) \to u(t)$ uniformly for a.e.~$t \in [0,T]$,} \\[5pt]
		\mbox{$Du_\epsilon \wsto Du$ weakly$^\ast$ in $L^\infty(\Omega_T,\R^n)$.}
	\end{array}
	\right.
	\label{eq:main_proof_convergence}
\end{equation}
It remains to show that $u$ is a variational solution to the Cauchy-Dirichlet problem associated with $f$ in the sense of Definition \ref{def:unconstrainded_solution}.
To this end, note that $u_\epsilon$ satisfies the variational inequality
\begin{align}
	\iint_{\Omega_\tau} f_\epsilon(t,Du_\epsilon) \,\dx\dt
	&\leq
	\iint_{\Omega_\tau} \partial_t v (v-u_\epsilon) \,\dx\dt
	+ \iint_{\Omega_\tau} f_\epsilon(t,Dv) \,\dx\dt
	\label{eq:main_proof_var_ineq} \\
	&\phantom{=}
	+\tfrac12 \| v(0) - u_o \|_{L^2(\Omega)}^2
	-\tfrac12 \| (v - u_\epsilon)(\tau) \|_{L^2(\Omega)}^2 \nonumber
\end{align}
for any $\tau \in [0,T] \cap \R$ and any comparison map $v \in K^\infty_{u_o}(\Omega_\tau)$ with $\partial_t v \in L^2(\Omega_\tau)$.
In the following, we pass to the limit $\epsilon \downarrow 0$ in \eqref{eq:main_proof_var_ineq}.
In order to treat the left-hand side, we rewrite
$$
	\iint_{\Omega_\tau} f_\epsilon(t,Du_\epsilon) \,\dx\dt
	=
	\iint_{\Omega_\tau} f(t,Du_\epsilon) \,\dx\dt
	+ \iint_{\Omega_\tau} f_\epsilon(t,Du_\epsilon) - f(t,Du_\epsilon) \,\dx\dt.
$$
By \eqref{eq:main_proof_convergence}$_3$ and Lemma \ref{lem:lower_semicontinuity} we obtain that
$$
	\iint_{\Omega_\tau} f(t,Du) \,\dx\dt
	\leq
	\liminf_{\epsilon \downarrow 0} \iint_{\Omega_\tau} f(t,Du_\epsilon) \,\dx\dt.
$$
Further, for $M := \max\{ Q, \| Du_o \|_{L^\infty(\Omega,\R^n)} \}$ we find that
\begin{align*}
	\bigg| \iint_{\Omega_\tau} f_\epsilon(t,Du_\epsilon) - f(t,Du_\epsilon) \,\dx\dt \bigg|
	\leq
	|\Omega| \int_0^\tau \sup_{|\xi| \leq M} |f_\epsilon(t,\xi) - f(t,\xi)| \,\dt
	\to 0
\end{align*}
as $\epsilon \downarrow 0$ by means of Lemma \ref{lem:Steklov_convergence}.
Joining the preceding two estimates yields
\begin{equation}
	\iint_{\Omega_\tau} f(t,Du) \,\dx\dt
	\leq
	\liminf_{\epsilon \downarrow 0} \iint_{\Omega_\tau} f_\epsilon(t,Du_\epsilon) \,\dx\dt.
	\label{eq:main_proof_aux1}
\end{equation}
Next, by \eqref{eq:main_proof_convergence}$_1$ we have that
\begin{equation}
	\iint_{\Omega_\tau} \partial_t v (v-u) \,\dx\dt
	=
	\liminf_{\epsilon \downarrow 0}
	\iint_{\Omega_\tau} \partial_t v (v-u_\epsilon) \,\dx\dt.
	\label{eq:main_proof_aux2}
\end{equation}
For the second term on the right-hand side of \eqref{eq:main_proof_var_ineq}, by Lemma \ref{lem:Steklov_convergence} we conclude that
\begin{equation}
	\bigg| \iint_{\Omega_\tau} f_\epsilon(t,Dv) - f(t,Dv) \,\dx\dt \bigg|
	\leq
	|\Omega| \int_0^\tau \sup_{|\xi| \leq M} |f_\epsilon(t,\xi) - f(t,\xi)| \,\dt
	\to 0
	\label{eq:main_proof_aux3}
\end{equation}
as $\epsilon \downarrow 0$.
Finally, \eqref{eq:main_proof_convergence}$_2$ shows that
\begin{equation}
	\| (v - u)(\tau) \|_{L^2(\Omega)}^2
	=
	\lim_{\epsilon \downarrow 0} \| (v - u_\epsilon)(\tau) \|_{L^2(\Omega)}^2
	\label{eq:main_proof_aux4}
\end{equation}
for a.e.~$\tau \in [0,T]$.
Collecting \eqref{eq:main_proof_aux1} -- \eqref{eq:main_proof_aux4}, we infer that
\begin{align*}
	\iint_{\Omega_\tau} f(t,Du) \,\dx\dt
	&\leq
	\iint_{\Omega_\tau} \partial_t v (v-u) \,\dx\dt
	+ \iint_{\Omega_\tau} f(t,Dv) \,\dx\dt\\
	&\phantom{=}
	+\tfrac12 \| v(0) - u_o \|_{L^2(\Omega)}^2
	-\tfrac12 \| (v - u)(\tau) \|_{L^2(\Omega)}^2
\end{align*}
for a.e.~$\tau \in [0,T]$ and any $v \in K^\infty_{u_o}(\Omega_\tau)$ with $\partial_t v \in L^2(\Omega_\tau)$.
In particular, this implies that $u \in C^0([0,T];L^2(\Omega))$, see Lemma \ref{lem:time_continuity}.
Therefore, we have that $u \in K^\infty_{u_o}(\Omega_T)$ is a variational solution associated with the integrand $f$ in the sense of Definition \ref{def:unconstrainded_solution}.
Finally, by the comparison principle in Theorem \ref{thm:comparison principle}, $u$ is unique.
This concludes the proof of Theorem \ref{thm:main_existence}.

\section{Continuity in time (Proof of Theorem \ref{thm:regularity})}
\label{sec:time}
To prove Theorem \ref{thm:regularity}, we begin by verifying that the unique variational solution $u$ to the Cauchy-Dirichlet problem associated with \eqref{eq:pde} and $u_o$ in $\Omega_T$ is a weak solution to \eqref{eq:pde} in $\Omega_T$.  To this end, let $\varphi \in C_0^\infty(\Omega_T)$ be a test function. We set $v_h := [u]_h + s[\varphi]_h$, where in the convolution we use the starting values $u_o$ and $\varphi(0) = 0$ for $u$ and $\varphi$, respectively. Using $v_h$ as a comparison function in \eqref{eq:var_ineq_unconstrained} and omitting the boundary term at $T$, we obtain that
\begin{equation}
\label{eq:whatever}
	0 \leq \iint_{\Omega_T} \partial_t v_h (v_h - u)\,\dx\dt + \iint_{\Omega_T} f(t, Dv_h)-f(t,Du)\,\dx\dt.
\end{equation}
Since by \eqref{eq:gradient_bound} we have that
\begin{align*}
	\|Dv_h\|_{L^\infty(\Omega_T,\R^n)}
	&\leq
	\|Du_o\|_{L^\infty(\Omega,\R^n)} + \|Du\|_{L^\infty(\Omega_T,\R^n)} + \|D\varphi\|_{L^\infty(\Omega_T,\R^n)} \\
	&\leq
	2 \|Du_o\|_{L^\infty(\Omega,\R^n)} + Q + \|D\varphi\|_{L^\infty(\Omega_T,\R^n)},
\end{align*}
it follows from \eqref{eq:dominating_f} that the sequence of mappings $(x,t) \mapsto f(t, Dv_h(x,t))$ has an integrable dominant independent of $h$. Therefore by the dominated convergence theorem, we conclude that
$$
	\lim_{h \downarrow 0} \iint_{\Omega_T} f(t,Dv_h) \,\dx\dt
	=
	\iint_{\Omega_T} f(t, Du+sD\varphi) \,\dx\dt.
$$
Further, by integration by parts and the convergence assertions from Lemmas \ref{lem:time_mollification} and \ref{lem:time_mollification_2}, we find that
\begin{align*}
	\iint_{\Omega_T} & \partial_t v_h(v_h-u)\,\dx\dt \\
		& = \iint_{\Omega_T} \partial_t [u]_h([u]_h-u) + s\partial_t [u]_h[\varphi]_h + s\partial_t[\varphi]_h([u]_h + s[\varphi]_h -u)\,\dx\dt \\
		& = \iint_{\Omega_T} \tfrac{1}{h}(u-[u]_h)([u]_h-u) - s \partial_t[\varphi]_h u \,\dx\dt\\
		& \phantom{=} + \int_\Omega s[u]_h[\varphi]_h(T) + \tfrac{s^2}{2} [\varphi]_h^2 (T)\,\dx \\
		& \leq -\iint_{\Omega_T} s\partial_t[\varphi]_hu\,\dx\dt + \int_\Omega s[u]_h[\varphi]_h(T) + \tfrac{s^2}{2} [\varphi]_h^2 (T)\,\dx \\
		&\to
		-\iint_{\Omega_T} s\partial_t \varphi u\,\dx\dt
\end{align*}
in the limit $h \downarrow 0$.
Thus, letting $h \downarrow 0$ in \eqref{eq:whatever} and dividing by $s$ we deduce that
\begin{align*}
	\iint_{\Omega_T} u \partial_t \varphi  \,\dx\dt & \leq \iint_{\Omega_T} \tfrac{1}{s} (f(t, Du+sD\varphi)-f(t,Du)) \,\dx\dt\\
		& = \iint_{\Omega_T} \int_0^1 D_\xi f(t, Du +s\sigma D\varphi)\cdot D\varphi\,\dsigma \dx\dt.
\end{align*}
Finally, observe that by the gradient bound \eqref{eq:gradient_bound} and the assumption \eqref{eq:uniform_lipschitz}, the integrand at the right-hand side of the above inequality is bounded. Thus we may let $s \rightarrow 0$ to obtain that
\begin{equation*}
	\iint_{\Omega_T} u \partial_t \varphi \,\dx\dt \leq \iint_{\Omega_T} D_\xi f(t, Du)\cdot D\varphi \,\dx\dt.
\end{equation*}
The reverse inequality follows by replacing $\varphi$ by $-\varphi$.

Consider cylinders of the form
\begin{equation*}
	Q_r := B_r(x_0) \times (t_0 - r^2, t_0 + r^2) \cap \Omega_T
\end{equation*}
where $(x_0, t_0) \in \overline{\Omega_T}$ and $r>0$. We show that $u$ satisfies the Poincar\'e inequality
\begin{align}
	\biint_{Q_r} &|u - (u)_{Q_r}|^2 \,\dx\dt \nonumber \\
	&\leq  C(n, \Omega) r^2 \bigg(\biint_{Q_r} |Du|^2 \,\dx\dt + \sup_{(x,t) \in Q_r} |D_\xi f(t, Du(x,t))|^2\bigg)
	\label{eq:time_reg 1}
\end{align} 
for all small $r>0$, where the mean value of $u$ over $Q_r$ is denoted by
$$
	(u)_{Q_r}
	:=
	\biint_{Q_r} u \,\dx\dt.
$$
Thus the gradient bound \eqref{eq:gradient_bound} together with condition \eqref{eq:uniform_lipschitz} yields
\begin{equation}
	\biint_{Q_r} |u - (u)_{Q_r}|^2 \,\dx\dt \leq C\big(n, \Omega, Q, \|Du_o\|_{L^\infty(\Omega,\R^n)}, f\big) r^2
\end{equation}
for all $r > 0$. The claim then follows from \cite[Theorem 3.1]{DaPrato}.

To prove \eqref{eq:time_reg 1},  we first note that since $\Omega$ is a convex domain, there exist positive constants $R(\Omega)$ and $C(\Omega)$ such that for any $r\in(0,R)$ and $x_0 \in \overline \Omega$, the set $\Omega \cap B_r(x_0)$ contains a ball of radius $r/C(\Omega)$. Then we assume that $Q_r$ with $r<R$ is given and denote $B_r := B_r(x_0)$, $t_1 := \max(t_0 - r^2, 0)$, $t_2 := \min(t_0 + r^2, T)$ so that $Q_r = (B_r \cap \Omega )\times (t_1, t_2)$. We fix a non-negative weight function $\eta \in C_0^\infty (B_r\cap\Omega)$ such that
\begin{equation*}
\bint_{B_r\cap\Omega} \eta \,\dx = 1 \quad \text{and} \quad \|\eta\|_{L^\infty(\Omega)} + r\|D\eta\|_{L^\infty(\Omega;\mathbb{R}^n)} \leq c(n,\Omega).
\end{equation*}
For the second assertion, we have used that $B_{r}\cap\Omega$ contains a ball of size $r/C(\Omega)$.
Since $B_{r}\cap\Omega$ is convex, the Poincar\'e inequality 
\begin{equation*}
	\int_{B_{r}\cap\Omega}\left|v-(v)_{B_{r}\cap\Omega}\right|^{2}\dx\leq\frac{r^{2}}{\pi^{2}}\int_{B_{r}\cap\Omega}\left|Dv\right|^{2}\dx
\end{equation*}
holds for any $v\in W^{1,2}(B_{r}\cap\Omega)$, see for example \cite{bebendorf}. An application of H\"older's and Minkowski's inequalities on the above further yields
\begin{equation}
	\label{eq:weighted_poincare}
	\bint_{B_r\cap\Omega}|v-(v\eta)_{B_r\cap\Omega}|^2\,\dx \leq c r^2 \bint_{B_r\cap\Omega}|Dv|^2\,\dx
\end{equation}
with a constant $c=c(n,\Omega)$.
We denote the weighted mean of $u$ at time $t$ by
\begin{equation*}
u_\eta(t) := \bint_{B_r\cap\Omega}u(x,t)\eta(x)\,\dx	
\end{equation*}
and decompose the left-hand side of \eqref{eq:time_reg 1} as follows
\begin{align*}
	\biint_{Q_r} & |u-(u)_{Q_r}|^2\,\dx\dt \\
	 		 &  \leq c\, \bint_{t_1}^{t_2}\bint_{B_r \cap\Omega}|u_\eta(t)-(u)_{Q_r}|^2\,\dx\dt + c\, \bint_{t_1}^{t_2}\bint_{B_r\cap\Omega}|u(x,t)-u_\eta(t)|^2\,\dx\dt  \\
			& = c\,\bint_{t_1}^{t_2} \Big|\bint_{t_1}^{t_2}u_\eta(t)-u_\eta(s)\,\ds+\bint_{t_1}^{t_2}u_\eta(s)\,\ds-(u)_{Q_r}\Big|^2\,\dt \\
			& \phantom{=} + c\, \bint_{t_1}^{t_2}\bint_{B_r\cap\Omega}|u(x,t)-u_\eta(t)|^2\,\dx\dt \\
			& \leq  c\, \bint_{t_1}^{t_2} \bint_{t_1}^{t_2} |u_\eta(t)-u_\eta(s)|^2 \,\ds\dt + c\, \Big|\bint_{t_1}^{t_2}u_\eta(s)\,\ds-(u)_{Q_r}\big|^2\\
			& \phantom{=} + c\, \bint_{t_1}^{t_2}\bint_{B_r\cap\Omega}|u(x,t)-u_\eta(t)|^2\,\dx\dt \\
			& =: c\, (I_1 + I_2 + I_3).
\end{align*}
To estimate $I_3$, we apply \eqref{eq:weighted_poincare} to obtain that
\begin{equation*}
	I_3 \leq c(n,\Omega) r^2\bint_{B_r\cap\Omega} |Du|^2\,\dx.
\end{equation*}
The same estimate holds for $I_2$ since by H\"older's inequality we have that
\begin{equation*}
	I_2 = \Big|\bint_{t_1}^{t_2}\bint_{B_r\cap\Omega}u_\eta(s)-u(x,s)\,\dx\ds\Big|^2 \leq I_3.
\end{equation*}
To estimate $I_1$, let $\tau_1, \tau_2 \in (t_1, t_2)$ with $\tau_1 < \tau_2$. As shown above, $u$ is a weak solution to \eqref{eq:pde}. That is, abbreviating $F(x,t) := D_\xi f(t, Du(x,t))$, we find that
\begin{equation*}
\int_0^T\int_\Omega u\partial_t\varphi - F\cdot D\varphi\,\dx\dt = 0\quad \text{for all } \varphi \in W^{1,\infty}_0 (\Omega_T).
\end{equation*}
Fix $\delta>0$ and consider
\[
\psi_\delta(t):=\begin{cases}
0, & t\in(0,\tau_{1}-\delta],\\
\frac{1}{2\delta}(t-(\tau_{1}-\delta)), & t\in(\tau_{1}-\delta,\tau_{1}+\delta),\\
1, & t\in[\tau_{1}+\delta,\tau_{2}-\delta],\\
1-\frac{1}{2\delta}(t-(\tau_{2}-\delta)), & t\in(\tau_{2}-\delta,\tau_{2}+\delta),\\
0, & t\in[\tau_{2}+\delta,T).
\end{cases}
\]
Using the test function $\varphi(x,t):=\eta\psi_\delta$ in the weak Euler-Lagrange equation yields
\begin{align*}
	0 &= \int_0^T \int_\Omega u\eta\partial_t\psi_\delta - \psi_\delta F \cdot D\varphi \,\dx\dt \\
		& = \bint_{\tau_1-\delta}^{\tau_1+\delta} \int_{B_r\cap\Omega} u\eta \,\dx\dt-\bint_{\tau_2-\delta}^{\tau_2+\delta} \int_{B_r\cap\Omega} u\eta \,\dx\dt+\int_{\tau_1-\delta}^{\tau_2+\delta}\int_{B_r\cap\Omega} \psi_\delta F \cdot D\eta \, \dx\dt.
\end{align*}
Passing to the limit $\delta \downarrow 0$, the preceding inequality implies that
\begin{align*}
	|u_\eta(\tau_1) - u_\eta(\tau_2)| & \leq \int_{\tau_1}^{\tau_2} \bint_{B_r\cap\Omega} |F\cdot D\eta|\,\dx\dt \\
						& \leq (\tau_2 - \tau_1) \|D\eta\|_{L^\infty}(\Omega,\mathbb{R}^n)\sup_{(x,t) \in Q_r } |F(x,t)| \\
						& = 2c(n,\Omega) r \sup_{(x,t) \in Q_r} |D_\xi f(t, Du(x,t))|
\end{align*}
holds true for almost every $\tau_1, \tau_2 \in (t_1,t_2)$.
In the last inequality, we used that $\tau_2 - \tau_1 \leq t_2 - t_1 \leq 2r^2$. Thus
\begin{equation}
I_1 \leq c(n,\Omega) r^2\sup_{(x,t) \in Q_r} |D_\xi f(t, Du(x,t))|^2.
\end{equation} Inequality \eqref{eq:time_reg 1} now follows by combining the estimates of $I_1, I_2$ and $I_3$.

\section*{Acknowledgements}
Jarkko Siltakoski was funded by the Magnus Ehrnrooth foundation.

\end{document}